\documentclass[10pt,journal]{IEEEtran}

\IEEEoverridecommandlockouts    

\usepackage{hyperref}
\usepackage{graphicx,color}
\graphicspath{{./IMG/}{images/}{./IMG/SAGEMATH_PLOTS/}{./figures_FO/}}
\usepackage{amsmath}
\usepackage{amssymb}
\usepackage{mathrsfs}
\usepackage{subfigure}
\usepackage{url}
\usepackage{booktabs}
\usepackage{array}
\usepackage[table]{xcolor}
\usepackage{mathtools} 		
\usepackage{epstopdf}
\usepackage{cite}
\usepackage[vlined,ruled,linesnumbered]{algorithm2e}
\usepackage{upgreek}
\usepackage{dsfont}
\usepackage[figurename=Fig.,font=footnotesize]{caption}
\usepackage{parskip}
\usepackage[shortlabels]{enumitem}

\usepackage{amsthm}
\theoremstyle{definition}
\newtheorem{theorem}{Theorem}  

\newtheorem{definition}{Definition}

\newtheorem{proposition}[theorem]{Proposition}
\newtheorem{problem}{Problem}
\newtheorem{remark}{Remark}
\newtheorem{assumption}{Assumption}
\newtheorem{example}{Example}

\newcommand{\mc}{\mathcal}

\newcommand{\im}{\operatorname{Im}}
\renewcommand{\ker}{\operatorname{Ker}}

\newcommand{\real}{\mathbb{R}}

\newcommand{\naturalpos}{\mathbb{N}_{>0}}
\newcommand{\naturalnneg}{\mathbb{N}_{\geq 0}}

\newcommand{\realnneg}{\real_{\geq 0}}

\newcommand{\complex}{\mathbb{C}}
\newcommand{\tsp}{\mathsf{T}}
 
\newcommand{\inv}{{\negat 1}} 
\newcommand{\negat}{\scalebox{0.75}[.9]{\( - \)}}
\newcommand*{\QEDB}{\hfill\ensuremath{\square}}
\newcommand{\QEDBB}{\tag*{$\square$}}
\newcommand*{\QEDBL}{\hfill\ensuremath{\blacksquare}}

\usepackage{pifont} 
\renewcommand{\ast}{\text{\ding{86}}}

\newcommand{\map}[3]{#1: #2 \rightarrow #3}

\newcommand{\until}[1]{\{1,\dots,#1\}}
\newcommand{\norm}[1]{\Vert #1 \Vert}

\DeclareMathAlphabet{\mymathbb}{U}{BOONDOX-ds}{m}{n}

\usepackage{colonequals}

\def\symmetric{\mathbb{S}}
\def\defeq{\colonequals}
\newcommand{\bmat}[1]{\begin{bmatrix}#1\end{bmatrix}}

\title{\LARGE\textbf{Feedback Optimization of Dynamical Systems in Time-Varying Environments: An Internal Model Principle Approach
}}

\begin{document}

\author{Gianluca Bianchin, {\it IEEE Member} \and\qquad Bryan Van Scoy, {\it IEEE Member}\thanks{%
G.~Bianchin is with the ICTEAM institute and the Department of Mathematical Engineering (INMA) at the University of Louvain, Belgium.
Email: \texttt{gianluca.bianchin@uclouvain.be}}%
\thanks{
B.~Van~Scoy is with the Dept. of Electrical and Computer Engineering, Miami University, OH~45056, USA. Email: \texttt{bvanscoy@miamioh.edu}%
} \hspace{-1cm}}

\maketitle

\thispagestyle{plain}\pagestyle{plain}

\begin{abstract}
Feedback optimization has emerged as a promising approach for regulating 
dynamical systems to optimal steady states that are implicitly defined by 
underlying optimization problems. Despite their effectiveness, existing methods 
face two key limitations: 
(i) reliable performance is restricted to time-invariant or slowly varying 
settings, and (ii) convergence rates are limited by the need for the controller 
to operate orders of magnitude slower than the plant.
These limitations can be traced back to the reliance of existing techniques on 
numerical optimization algorithms.
In this paper, we propose a novel perspective on the design of feedback 
optimization algorithms, by framing these objectives as an output regulation 
problem. 
We place particular emphasis on time-varying optimization problems, and show 
that an algorithm can track time-varying optimizers if and only if it 
incorporates a model of the temporal variability inherent to the optimization -- 
a requirement that we term the \textit{internal model principle of feedback 
optimization.}
Building on this insight, we introduce a new design methodology that couples 
output-feedback stabilization with a control component that drives the system 
toward the critical points of the optimization problem. This framework enables 
feedback optimization algorithms to overcome the classical limitations of slow 
tracking and poor adaptability to time variations.
\end{abstract}

\vspace{-.5cm}


\section{Introduction}
\label{sec:intro}

Feedback optimization 
techniques~\cite{MC-ED-AB:20,AH-SB-GH-FD:20,AJ-ML-PB:09,SM-AH-SB-GH-FD:18,LL-ZN-EM-JS:18,GB-JC-JP-ED:21-tcns} are concerned with the problem of 
controlling dynamical systems to an optimal steady-state point, where 
optimality is implicitly defined by an underlying mathematical optimization
problem~\cite{AJ-ML-PB:09}. 
This framework has been successfully applied in a variety of domains, including 
optimal scheduling in communication networks, resource allocation in power 
systems, transportation system optimization, and the operation of industrial 
control processes~\cite{AH-ZH-SB-GH-FD:24}.
The foundational design approach underlying these methods begins with an 
established optimization algorithm~\cite{YN:18}, and in suitably adapting it to 
incorporate feedback measurements in place of unknown 
quantities~\cite{AJ-ML-PB:09,AH-ZH-SB-GH-FD:24}. This modification results in a 
feedback loop between the optimization algorithm and the dynamical system.
Although numerous feedback optimization methods have been developed, these 
techniques inherently carry over the limitations of the optimization algorithms 
from which they originate. 
Such limitations include fundamental bounds on the achievable convergence 
rate~\cite{MB-FD:25}, limited robustness to 
uncertainty~\cite{GB-JIP-ED:20-automatica}, and the inability to exactly track 
optimizers when in time-varying 
settings~\cite{GB-JIP-ED:20-automatica,GB-JC-JP-ED:21-tcns}.

\begin{figure}[tb]
\centering
\includegraphics[width=.8\columnwidth]{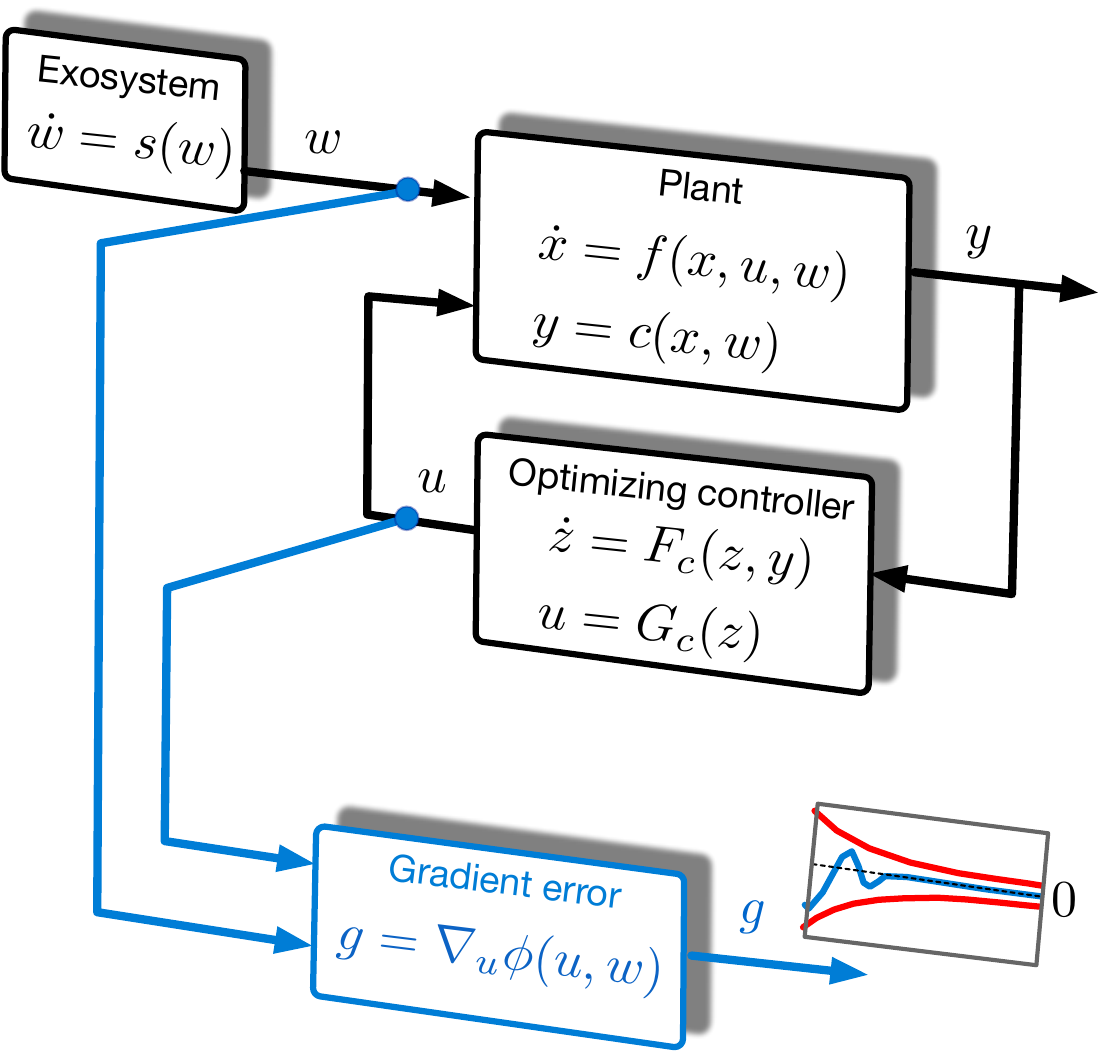}
\caption[]{Architecture of the feedback optimization scheme studied in this 
work.  Given a plant 
subject to a time-varying disturbance,
the goal is to design a control algorithm
having access only to measurements $y(t)$ and generating a 
sequence of control actions $u(t)$ such that the interconnection is stabilized 
around an equilibrium and the gradient signal $g(t) \to 0$~as~${t \to \infty}$. 
Blue indicates quantities that are unmeasurable by the controller; in 
particular, the controller has no access to the 
signal $g(t)$ to be regulated to zero. See~\eqref{eq:interconnected_system}.
}
\label{fig:block_diagram_FO}
\vspace{-.5cm}
\end{figure}

In this work, we approach the design of feedback optimization algorithms from a 
novel perspective: by formulating it as an output regulation 
problem~\cite{ED:76,BF-WW:76,BF:77,AI-CIB:90,CD-CL:85}. 
Specifically, in this setting, the signal to be regulated corresponds to the 
(unmeasurable) gradient error -- see Fig.~\ref{fig:block_diagram_FO} for an 
illustration.
By establishing this connection, we are able not only to design novel feedback 
optimization algorithms that are more general than those derived from classical 
optimization methods, but also to study an entire class of algorithms and 
establish fundamental limitations for the whole class. This class includes, as 
special cases, many well-established techniques~\cite{MC-ED-AB:20,AH-SB-GH-FD:20,AJ-ML-PB:09,SM-AH-SB-GH-FD:18,LL-ZN-EM-JS:18,GB-JC-JP-ED:21-tcns}.
Our approach introduces three key practical innovations compared to the 
existing literature:  
(i) under suitable system knowledge, our methods enable exact tracking of 
optimal steady states despite arbitrary time-varying disturbances -- whereas 
existing techniques can only achieve approximate tracking 
(see~\cite{GB-JC-JP-ED:21-tcns} and references therein);  
(ii) our framework does not require plant pre-stabilization, as it 
simultaneously addresses both stabilization and setpoint tracking; and  
(iii) it removes the need for a time-scale separation between the plant and the 
controller, thereby improving the achievable rate of convergence.

{\bf Related work.}
The field of feedback optimization stems from the seminal 
works~\cite{AJ-ML-PB:09,FB-HD-CE:12}, which considered the objective of 
steering the outputs of a plant to a steady-state that solves a convex 
optimization problem. 
While early works focused on static plants, a more recent body of 
literature~\cite{LL-ZN-EM-JS:18,LL-JS-EM:20,MC-ED-AB:20,AH-SB-GH-FD:20,AJ-ML-PB:09,SM-AH-SB-GH-FD:18,GB-JC-JP-ED:21-tcns} extends feedback optimization to 
dynamic plants, addressing closed-loop stability and exponential convergence 
rates~\cite{GB-JC-JP-ED:21-tcns,GB-JIP-ED:20-automatica}, including in 
sampled-data settings~\cite{GB-DL-MD-SB-JL-FD:21}. 
Constrained variants of the problem have also been explored, particularly using 
primal-dual dynamics and projection-based 
methods~\cite{MC-ED-AB:20,AH-SB-GH-FD:20,GB-JC-JP-ED:21-tcns}. An alternative 
approach employs barrier function 
techniques~\cite{YC-LC-JC-ED:23,GD-PM-JC-MH-25}.
Nonconvex problems have been considered 
in~\cite{YT-ED-AB-SH:18,VH-AH-LO-SB-FD:20}, gradient-free methods 
in~\cite{ZH-SB-JH-FD-XG:23}, distributed approaches 
in~\cite{GC-NM-GN:24,AM-GB:25}. Recent developments have sought to relax the 
reliance on time-scale separation~\cite{MB-FD:25}.

A central contribution of this work is the reformulation of the feedback 
optimization problem as an output regulation problem; this connects our work 
with the literature on output regulation. This 
field originated in the 1970s with foundational work on linear 
systems~\cite{BF-WW:76,ED:76,BF:77}, and was subsequently extended to 
nonlinear systems through both local~\cite{AI-CIB:90} and 
global~\cite{CB-AI-LP:03} analysis techniques. In recent years, output 
regulation has seen renewed interest through the lens of modern control 
methods; see, for example,~\cite{AI-LM-LP:12} and the 
recent tutorial~\cite{JH-AI-LM-MM-ES-WW:18}. Connections between optimal steady-state 
control and output regulation have been also recognized in~\cite{LL-JS-EM:18}.


Finally, since the problem studied in this work is a time-varying 
optimization problem, it is naturally connected to the literature on 
online optimization~\cite{EH:16,AS-ED-SP-GL-GG:20}. Particularly 
relevant are the recent works~\cite{NB-RC-SZ:24,GB-BVS:24-arxiv,AW-IP-VU-IS:24}, which explore the use of internal models to address such problems.

{\bf Contributions.}
This paper features three main contributions. 
First, we demonstrate that the feedback optimization problem can be formulated 
as an output regulation problem (see 
Section~\ref{sec:reformulation_as_output regulation} and 
Definition~\ref{def:tracking}). 
This reformulation enables the use of tools from the output regulation literature 
to address the problem, allowing us to analyze feedback optimization algorithms as 
a broad class of methods -- rather than relying on ad hoc algorithmic constructs 
from optimization theory.
From a methodological perspective, the techniques we develop here extend existing 
approaches in several key directions:  
(i) they enable exact asymptotic tracking of a critical point even in the presence 
of arbitrary time-varying disturbances -- whereas prior work has been limited to 
constant disturbances;  
(ii) they allow for simultaneous plant stabilization and 
feedback optimization, while these two objectives have typically been treated 
separately in the literature; and  
(iii) they eliminate the need for a time-scale separation between the plant and 
the controller, which removes inherent limitations on the maximum 
achievable convergence rate.   To the best of our knowledge, all of these 
features are novel within the literature.
Second, we establish fundamental limitations for a broad class of feedback 
optimization methods (see Section~\ref{sec:necessary_conditions} and 
Theorem~\ref{thm:existence_dynamic_feedback}). Notably, we show that exact tracking 
of a critical point is possible if and only if the algorithm embeds a duplicated 
representation of the disturbance signal -- though this copy may be expressed in a 
different coordinate system (see
Theorem~\ref{thm:characterization_dynamic_feedback}).  We refer to this concept as 
an internal model principle, akin to its counterpart in time-varying 
optimization~\cite{GB-BVS:24-arxiv} and controls~\cite{BF-WW:76,JH-WW:84}.
Interestingly, since many existing feedback optimization methods in the literature 
are special cases of the general class of algorithms considered in this work (see 
Remark~\ref{rem:basig_feedback_opt}), our results imply that these methods achieve 
exact tracking only when the time-varying signals involved in the optimization 
problem are, in fact, constant. In all other cases, they necessarily yield only 
inexact tracking (see Remark~\ref{rem:internal_model_interpret_feedopt}).
To the best of our knowledge, this is the first work in the literature that 
formally proves the necessity of internal models -- a requirement that aligns with 
recent insights presented in~\cite{GB-BVS:24-arxiv}.
Third, we derive necessary and sufficient conditions for exact asymptotic 
tracking of a critical trajectory (see
Theorem~\ref{thm:characterization_dynamic_feedback}) and, by leveraging these, 
we introduce a novel design procedure for feedback optimization algorithms. 
Our design technique relies on a separation principle, which combines an 
observer to estimate unmeasurable dynamical states, and a static feedback 
control action, which steers the system toward the set of critical points (see 
Fig.~\ref{fig:separation_principle_structure}, Section~\ref{sec:alg_design} and 
Algorithm~\ref{alg:dynamic_feedback_design}).
%

{\bf Organization.}
Section~\ref{sec:prob_formulation} introduces the problem, which is reformulated 
as an output regulation problem in 
Section~\ref{sec:reformulation_as_output regulation}. 
Section~\ref{sec:necessary_conditions} presents necessary conditions for 
solvability and outlines the controller structure. Static and dynamic feedback 
optimization are addressed in Sections~\ref{sec:static_feedback} 
and~\ref{sec:dynamic_feedback}, respectively. Section~\ref{sec:alg_design} 
details the algorithm design, while Section~\ref{sec:extensions} discusses 
extensions to constrained problems. Numerical validations are provided in 
Section~\ref{sec:simulations}, and conclusions in Section~\ref{sec:conclusions}. 
Technical proofs appear in Appendix~\ref{sec:comlemetary_proofs}, and 
Appendix~\ref{sec:center_manifold} reviews relevant center manifold theory.

{\bf Notation.}
We let
$\complex_{<}\defeq\{s: \operatorname{Re} s<0\}$ and 
$\complex_{\geq}\defeq\{s: \operatorname{Re} s \geq 0\}.$ 
We denote the space of $n\times n$ symmetric real matrices by $\symmetric^n$. 
Given an open set $U,$ we say that $\map{f}{U}{\real}$ is of 
differentiability class $C^k$ if it has a $k\textsuperscript{th}$ derivative 
that is continuous in $U.$ 
%
Given $A \in \real^{n \times n},$ denote its eigenvalues by
$\lambda_j = a_j + ib_j, a_j, b_j\in \real,j \in \until n$ with corresponding 
(generalized) eigenvectors $w_j = u_j + iv_j, u_j, v_j\in \real^n$. The space 
$\mathscr{X}_<(A)\defeq\operatorname{span}\{u_j, v_j:a_j<0\}$
is the {\it stable subspace} of $A$, and 
$\mathscr{X}_{\geq}(A) \defeq\operatorname{span}\{u_j, v_j:a_j\geq 0\}$ is the 
{\it unstable subspace} of $A$. Given $B \in \real^{n \times m}$, the {\it 
controllable subspace} of $(A,B)$ is 
$\mc C(A,B) \defeq \sum_{i=1}^{n} \im(A^{i-1} B).$ 
The pair $(A,B)$ is {\it controllable} if $\mc C(A,B)=\real^n$ and 
{\it stabilizable} if $\mathscr{X}_{\geq}(A) \subseteq \mc C(A,B).$
Given $C \in \real^{q \times n}$, the {\it unobservable subspace} of $(C,A)$ is $\mc O^c(C,A) \defeq \cap_{i=1}^{n} \ker(CA^{i-1}).$ 
The pair $(C,A)$ is observable if $\mc O^c(C,A) = \emptyset$ and 
{\it detectable} if $\mc O^c(C,A) \cap \mathscr{X}_{\geq}(A) = \emptyset.$

\begin{table*}[t]
\caption{Comparison between our framework and existing methods in 
reference tracking~\cite{ED:76,BF-WW:76,BF:77,AI-CIB:90}, disturbance 
rejection~\cite{CD-CL:85}, and feedback 
optimization~\cite{AH-SB-GH-FD:20,GB-JC-JP-ED:21-tcns,MC-ED-AB:20}. 
See Remarks~\ref{rem:relation_problem_literature}, 
\ref{rem:basig_feedback_opt}, and \ref{rem:internal_model_interpret_feedopt}
for  discussions.
}
\setlength{\tabcolsep}{4pt}
\centering\begin{tabular}{l|c|c|c|c|c|}
  & $w(t)\neq 0$ &  Plant & Controller at & Automatic reference & $x$-dependent \\
    & \& time-varying  & stabilization & same timescale &  generation & objectives  \\
  \hline
  Reference tracking~\cite{ED:76,BF-WW:76,BF:77,AI-CIB:90} & & \checkmark & \checkmark & & \\
  \rowcolor{gray!25}Tracking and disturbance rejection~\cite{CD-CL:85} & \checkmark & \checkmark &\checkmark   & & \\
  Feedback optimization~\cite{AH-SB-GH-FD:20,MC-ED-AB:20,AH-ZH-SB-GH-FD:24} & & & & \checkmark & \\
  \rowcolor{gray!25} Time-varying feedback optimization~\cite{GB-JC-JP-ED:21-tcns} & \checkmark & & & \checkmark &   \\
  This work & \checkmark &\checkmark & \checkmark & \checkmark & \checkmark \\
  \hline
\end{tabular}
\label{tab:comparison_literature}
\end{table*}

\section{Problem formulation}
\label{sec:prob_formulation}

In this section, we formulate the problem studied and illustrate its 
applicability through a representative example.

\subsection{Problem statement}
We consider plants described by nonlinear systems of differential equations of 
the form:
\begin{align}\label{eq:plant}
\dot x(t) &= f(x(t),u(t),w(t)),\nonumber \\
y(t) &= c(x(t),w(t)),
\end{align}
where $t \in \realnneg$ denotes time, $x(t)\in X\subseteq\real^n$ is the state, $u(t) \in \real^m$ is the control input, 
$y(t) \in \real^q$ is the measurable output, and $w(t)\in W\subseteq\real^p$ is a disturbance. 
We assume that $f: X \times \real^m \times W \to \real^n$ 
and $c: X \times W \to \real^q$ are~$C^1$.

The signal $w(t)$ models exogenous disturbances that may affect the plant and/or 
measurements, possibly independently (see Section~\ref{sec:segway_robot}).
Consequently, we treat $w(t)$ as an unknown and unmeasurable 
time-varying signal in our analysis.

In this work, we study the problem of devising control actions that accomplish 
two goals: (i) locally exponentially stabilize the plant~\eqref{eq:plant}, and 
(ii) regulate~\eqref{eq:plant} to an optimal equilibrium point. 
The second objective is formalized by means of the following mathematical 
optimization problem:
\begin{align}\label{eq:optimization}
\underset{u \in \real^m, x \in X}{\text{minimize}} ~~~& \phi_0(u,x),\nonumber\\
\text{subject to:} ~~ & 0  = f(x,u,w(t)),
\end{align}
where $\phi_0: \real^m \times X \rightarrow \real$. 
This optimization formalizes an equilibrium-selection problem,
which seeks to select an equilibrium input $u$ and state $x$ 
for~\eqref{eq:plant}
that minimize the loss $\phi_0(u,x)$, which 
quantifies performance at the equilibrium\footnote{Our approach also 
applies to more general constraints than merely equilibrium selection (such 
as predictive control), provided that any state that satisfies the 
constraints can be written as a function of the control input and exogenous 
noise, as in Assumption~\ref{as:steadyStateMap}.}. 

From an optimization perspective, \eqref{eq:optimization} is an optimization 
problem that is {\it parametrized} by the exogenous signal $w(t)$. This aspect 
has two important implications: first, because $w(t)$ is unknown and 
unmeasurable, solutions to~\eqref{eq:optimization} cannot be computed 
explicitly by an optimization solver (since a numerical value for $w(t)$ 
cannot be substituted into~\eqref{eq:optimization} to solve the optimization); 
second, because $w(t)$ is a time-varying signal, \eqref{eq:optimization} 
defines a {\it sequence} of optimization problems (one at each time $t$), and 
thus the problem of regulating the plant to solutions 
of~\eqref{eq:optimization} involves also {\it tracking these solutions over 
time.} 
For these reasons, our control objective cannot be accomplished using standard 
control approaches such as reference generation plus reference 
tracking~\cite{ED:76,BF-WW:76,BF:77,AI-CIB:90}. 
We informally state the problem of interest as follows. 

\begin{problem}[\textbf{\textit{Feedback optimization algorithm design -- informal}}]
\label{prob:feedback_opt_informal}
Construct, when possible, a control algorithm such that: 
(i) the equilibrium of~\eqref{eq:plant} with $w(t)\equiv 0$ is locally exponentially 
stable, and (ii) the states and inputs of \eqref{eq:plant} converge, with 
exactly zero asymptotic error, to a solution of~\eqref{eq:optimization} for any time-varying signal $w(t)$ in some class.~
\QEDB\end{problem}

We refer to a control algorithm that achieves these goals as a 
\textit{feedback-optimization algorithm}, in line with the 
counterparts in~\cite{GB-JC-JP-ED:21-tcns,AH-SB-GH-FD:20,GC-NM-GN:24}.
We stress that existing methods are capable of tracking solutions 
of~\eqref{eq:optimization} only {\it inexactly}, unless $w(t)$ is a 
constant signal. In contrast, our goal is to design control 
algorithms capable of achieving exactly zero asymptotic error when 
$w(t)$ is time-varying.

We conclude this discussion by relating our problem with 
the existing literature in Remark~\ref{rem:relation_problem_literature}.

\begin{remark}[\textbf{\textit{Relationship of~\eqref{eq:optimization} with 
models commonly studied in the literature}}]
\label{rem:relation_problem_literature}
The optimization objective~\eqref{eq:optimization} is related to the classical 
{\it (reference) tracking problem}, also called the {\it (output) regulation 
problem} or {\it servomechanism problem}~\cite{ED:76,BF-WW:76,BF:77,AI-CIB:90}, which consists of designing a controller for~\eqref{eq:plant} such that, with $w(t) \equiv 0$, the closed-loop system is stable and $\lim_{t \to \infty} y(t) - r(t) = 0$, where $r(t)$ is a given reference signal.  
More generally, \eqref{eq:optimization} is related to the {\it 
(reference) tracking and disturbance rejection 
problem}~\cite{CD-CL:85}, which extends the 
reference tracking problem by allowing $w(t)$ to be nonzero and time-varying. 
The optimization objective~\eqref{eq:optimization} is also related to the 
feedback optimization literature~\cite{AH-SB-GH-FD:20,GB-JC-JP-ED:21-tcns}, 
where controllers are designed by drawing inspiration from algorithmic 
optimization. See Table~\ref{tab:comparison_literature}.

The optimization~\eqref{eq:optimization} extends these existing 
approaches in at least four directions. First, in reference tracking 
problems, the reference signal to be tracked is prespecified, thus dictating 
the need for external reference generation mechanisms 
(e.g.,~\cite{SL:06}), which often require knowledge of $x(t)$ and 
$w(t),$  thus making these approaches unfeasible or suboptimal when 
these signals are unknown. 
Instead, the 
formulation~\eqref{eq:optimization} allows us to implicitly generate the 
reference to be tracked as the solution to a mathematical optimization 
problem. 
Second, both in feedback optimization and reference tracking, performance 
objectives or references to be tracked can be specified only in terms of the 
plant's output $y(t)$ and not in terms of the plant's state $x(t)$ 
(or, in other words, they can be applied to solve~\eqref{eq:optimization} only 
when the plant's output map is such that $y(t)=x(t)$). 
By allowing the loss function to depend on the state, \eqref{eq:optimization}
provides additional flexibility in applications with partially observed plants 
(i.e., where $x(t)$ cannot be directly estimated from $y(t)$); this 
aspect is illustrated in detail Section~\ref{sec:segway_robot}. 
Third, feedback optimization techniques require the plant to be 
prestabilized. Fourth, feedback optimization methods rely on a timescale 
separation between the (fast) plant and (slow) controller, and this poses 
limitations on the maximum attainable rate of convergence. 
Additional connections between our methods and existing techniques are 
established in Table~\ref{tab:comparison_literature}.
\QEDB\end{remark}
%

\subsection{Illustrative application: automatically optimize an unstable system subject to unknown disturbances}
\label{sec:segway_robot}
Consider a two-wheeled balancing robot moving on a surface as in Fig.~\ref{fig:segway}(a), and the objective of self-balancing this inverted-pendulum type system. For this task, we consider the kinematic model in 
Fig.~\ref{fig:segway}(b); letting $\theta$ be the rotation angle of the 
center of mass and $r$ the horizontal displacement of the wheels, we model the 
robot using the equation of motion:
\begin{align*}
J_e \ddot{\theta}(t)= m g \ell \sin \theta(t)-k \ell^2 \dot{\theta}(t) 
- m \ell \ddot r(t) \cos \theta(t) + w_x(t),
\end{align*}
where $J_e>0$ is the moment of inertia of the rod about its end, $m>0$ the 
mass, $\ell>0$ the distance from the center of gravity to the pivot, $g$ the 
gravitational acceleration, and we assumed the presence of a frictional force 
with coefficient $k>0$. The signal $w_x(t)$ is used to describe external 
disturbances or model discrepancies, such as uneven surface profiles or 
frictions.
Viewing the cart's acceleration $\ddot r(t)$ as the control input (i.e., 
letting $u(t)\defeq\ddot r(t)$), the problem of self-balancing the robot can be 
formulated as the following instance of~\eqref{eq:optimization}:
\begin{align}\label{eq:optimization_regulation_pendulum}
\underset{u, \theta, \dot \theta \in \real}{\text{minimize}} ~~~& 
\frac{1}{2} ( \theta^2 + \dot \theta^2),\nonumber\\
\text{subject to:} ~~ & 0 = \dot \theta, \nonumber\\
& 0 = m g \ell \sin \theta - m \ell u \cos \theta + w_x(t).
\end{align}
Suppose the robot is equipped with an integrated Inertial Measurement Unit 
(IMU) providing noisy measures of the angular rate of change $\dot \theta(t)$:
\begin{align}\label{eq:y_pendulum}
y(t) &= \dot \theta(t) + w_y(t),
\end{align}
where $w_y(t)$ models sensor noise. Because $y(t)$ does not include 
displacement information (i.e., it does not explicitly depend on $\theta(t)$), 
any method that achieves $y(t)\to 0$ as $t \to \infty$ (such as reference 
tracking or feedback optimization, see 
Remark~\ref{rem:relation_problem_literature}) will not guarantee 
that $\theta(t) \to 0,$ but only that $\dot \theta(t) \to 0.$
This limitation is illustrated in Fig.~\ref{fig:segway}(c), where 
a reference tracking technique~\cite{AI-CIB:90} is applied 
to this problem. In contrast, the techniques proposed in this work are capable 
of achieving $\theta(t) \to 0$ (see Fig.~\ref{fig:segway}(d)). 
See Section~\ref{sec:simulations} for additional details on this problem and a 
description of the techniques used.

\begin{figure}[t]
\centering \subfigure[]{\includegraphics[width=.4\columnwidth]{
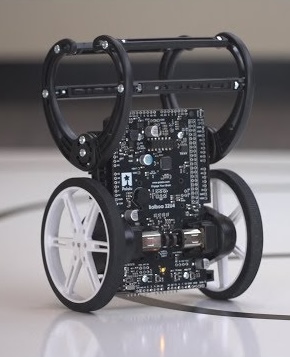}}
\hfill
\centering \subfigure[]{\includegraphics[width=.4\columnwidth]{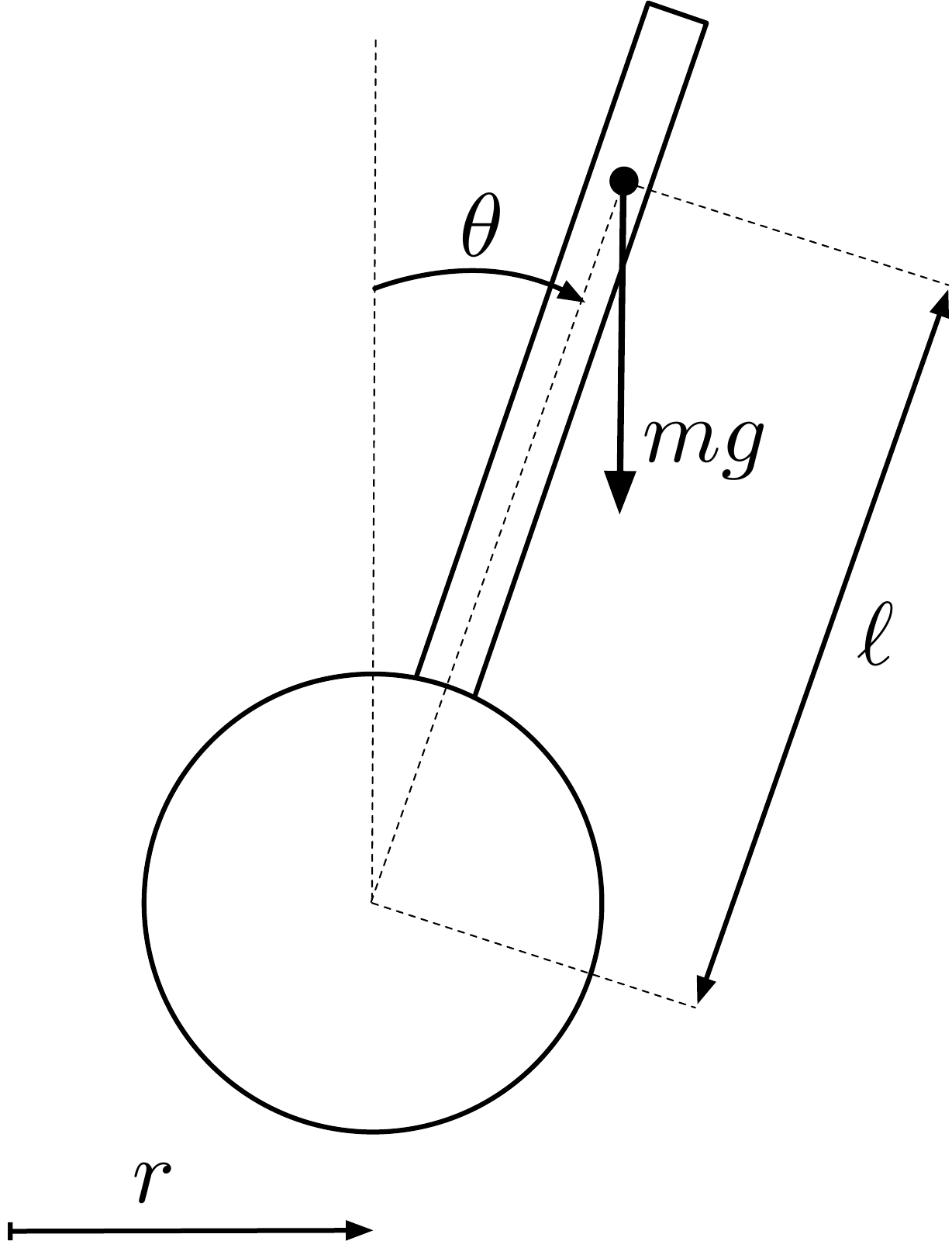}}\\
\centering \subfigure[]{\includegraphics[width=\columnwidth]{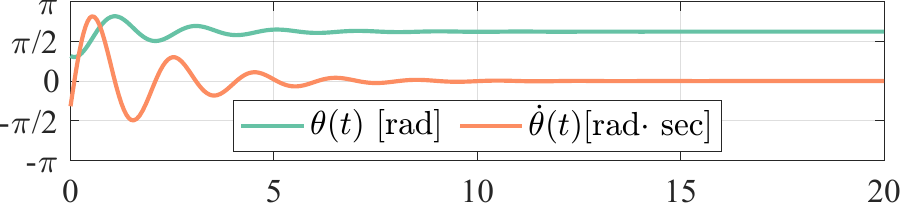}}\\
\centering \subfigure[]{\includegraphics[width=\columnwidth]{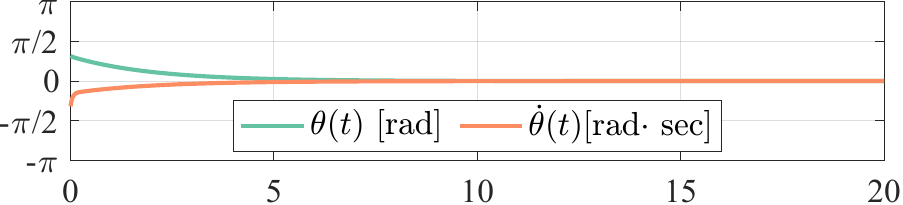}}
\caption{
(a) Illustration of the balancing robot discussed in 
Section~\ref{sec:segway_robot}.  
(b) Free-body diagram of the robot.  
(c) State evolution under control via a standard output regulation algorithm.  
(d) State evolution under the control strategy proposed in this work.  
The simulation demonstrates that traditional output regulation 
algorithms fail to stabilize the robot in the upright position -- 
i.e., to achieve $\theta(t) \to 0$ -- due to the inability to 
measure $\theta(t)$. 
Parameters used:  $\ell = 0.023$ [m],  $m = 0.316$ [Kg], 
$k=0.1$ [Kg/s], $g = 9.81$ [m/s$^2$], $J_e = 0.000444$ [Kg\,m$^2$].
For simplicity of the illustration, this simulation uses $w_x(t) \equiv 0$ and 
$w_y(t) \equiv 0$; see Section~\ref{sec:simulations} for additional simulations.
}
\vspace{-.5cm}
\label{fig:segway}
\end{figure}

\section{Technical reformulation of the problem: from feedback optimization to output regulation}
\label{sec:reformulation_as_output regulation}

In this section, we present a formal reformulation of 
Problem~\ref{prob:feedback_opt_informal}, connecting it with the output 
regulation framework.

\subsection{Standing assumptions}

We now present the assumptions on which our approach is based. We 
make the following standard assumption on the~loss. 

\begin{assumption}[\textbf{\textit{Properties of the objective function}}]
\label{as:lipschitz_convexity}
The map $(u,x) \mapsto \phi_0(u, x)$ is convex and differentiable, and $(u,x) \mapsto \nabla \phi_0(u,x)$ is Lipschitz continuous in $\real^m \times X.$
\end{assumption}

Convexity and smoothness are standard assumptions in 
optimization~\cite{EH:16}, which have been widely used in works on related 
problems~\cite{GB-JC-JP-ED:21-tcns,AH-SB-GH-FD:20,GC-NM-GN:24}.

We assume that the temporal variability of the disturbance $w(t)$ belongs to 
the following class. 

\begin{assumption}[\textbf{\textit{Class of temporal variabilities of the disturbance signal}}]
\label{as:exosystem}
There exists a $C^1$ vector field $s: W \rightarrow \real^p$ such that the disturbance signal $w(t)$ satisfies:
\begin{align}\label{eq:exosystem}
\dot{w}(t)=s(w(t)) \qquad \forall t\in\real_{\geq 0}.
\end{align}
Moreover, \eqref{eq:exosystem} has an equilibrium at $w = 0$ and its 
trajectories are bounded. 
\QEDB
\end{assumption}

We do not assume \textit{a priori} that $s(w)$ or $w(t)$
are known; see 
Sections~\ref{sec:static_feedback}--\ref{sec:dynamic_feedback}, where 
different assumptions on the knowledge of these two quantities are discussed. 
Assumption~\ref{as:exosystem} characterizes the class of temporal 
variabilities for the disturbance taken into account. 
This assumption is mild, as it only requires that $w(t)$ is deterministic, 
sufficiently smooth (so that its derivative is some $C^1$ function), and 
its trajectories remain bounded.  Following established terminology in the 
literature~\cite{AI-CIB:90,ED:76}, we refer to the model 
in~\eqref{eq:exosystem} as the {\it exosystem}.

We observe that asymptotically stable modes in $w(t)$ do not affect the 
optimization~\eqref{eq:optimization} when $t\to\infty$ (precisely, because 
these modes converge to $0$ as $t\to\infty$). Hence, without loss of 
generality, in what follows, we will assume that~\eqref{eq:exosystem} has 
no asymptotically stable modes.
For simplicity of the presentation, we also assume that $W$ is some 
neighborhood of 
the origin of $\real^p$. We put no restrictions on the size of this 
neighborhood (which is, e.g., allowed to be the entire space 
$W = \real^p$), and thus on the size of $w(t)$ nor of its temporal 
variation. Moreover, there is no 
restriction with asking that $W$ contains the origin because, if 
$w(t)$ takes values in the neighborhood of any other point, such a point
can be shifted to the origin through a time-invariant change of variables 
without altering the solutions of~\eqref{eq:optimization}. 

We make the following assumption on the plant to control.


\begin{assumption}[\textbf{\textit{Existence of a steady-state map}}]
\label{as:steadyStateMap}
There exists a unique $C^1$ function $h: \real^m \times W \to {X}$ such that
$f(h(u,w), u,w)=0,$ for any $u \in \real^m$ and $w \in W$. 
\QEDB\end{assumption}

In words, Assumption~\ref{as:steadyStateMap} guarantees that the 
plant~\eqref{eq:plant}, with constant inputs, admits a steady-state operating 
point.
Existence of such a steady-state map is guaranteed in most practical cases of 
interest~\cite{HKK:96}. For example, this condition is ensured when 
$\nabla_x f(x,u,w)$ is invertible in $\real^m \times W$; alternatively, it 
can be guaranteed by application of the implicit function 
theorem~\cite[Thm.~2]{HN-FF:22}, \cite[Thm.~6]{AA-SZ:19} to the 
equation ${f(x,u, w)=0.}$

Using Assumption~\ref{as:steadyStateMap}, the optimization 
problem~\eqref{eq:optimization} can be reformulated as an unconstrained problem:
%
%
\begin{align}\label{eq:optimization_unconstrained}
\underset{u \in \real^m}{\text{minimize}} ~~~& \phi_0(u,h(u, w(t)) = \phi(u,w(t)),
\end{align}
where $\phi(u,w) \defeq \phi_0(u,h(u, w))$ for $u \in \real^m$ and $w \in W.$ 
Next, we define a critical trajectory as an input signal that results in a vanishing gradient of the unconstrained problem.

\begin{definition}[\textbf{\textit{Critical trajectory}}]\label{defn:critical-trajectory}
The function $\map{u^\star}{\realnneg}{\real^m}$ is said to be a 
\textit{critical trajectory} of~\eqref{eq:optimization_unconstrained} if it 
satisfies:
\begin{align*}
0=\nabla_u \phi(u^\star(t),w(t)), \quad \forall t \in \realnneg. \QEDBB
\end{align*}
\end{definition}

Notice that, if $u^\star(t)$ is a critical trajectory 
of~\eqref{eq:optimization_unconstrained}, then 
$(u^\star(t),h(u^\star(t),w(t)) \in \real^m \times X$ is a critical trajectory 
of~\eqref{eq:optimization},
\begin{align*}
0=\left.\nabla_u \phi_0(u,h(u,w))\right\vert_{(u,w)=(u^\star(t),w(t))}, 
\quad \forall t \in \realnneg.
\end{align*}
We therefore refer to 
the critical trajectories of~\eqref{eq:optimization} and those 
of~\eqref{eq:optimization_unconstrained} interchangeably. 
For the optimization problem to be well-posed, we make the following 
assumption. 

\begin{assumption}[\textit{\textbf{Existence and continuity of critical trajectories}}]
\label{as:existence_continuity_critical_trajectory}
The optimization~\eqref{eq:optimization_unconstrained} admits a critical 
trajectory. Moreover, every critical trajectory is continuous. 
\QEDB\end{assumption}

Existence of a critical trajectory for~\eqref{eq:optimization} can be 
guaranteed under mild assumptions; for example, coercivity 
of the cost function (i.e., $\phi_0(u,x)\to \infty$ when 
$\Vert{u}\Vert \to \infty$), or by 
requiring that the search domain can be restricted to a compact set without 
altering the optimizers (by Weierstrass' theorem~\cite{RS:96}). 
Continuity of the critical trajectories can also be ensured under standard 
assumptions; for example, by requiring that 
the loss $\phi_0(u,x)$ is continuous in $x$ (by Berge's theorem~\cite{RS:96}).

\subsection{Controller structure}

To address Problem~\ref{prob:feedback_opt_informal}, we search 
within a class of candidate control algorithms that do not require 
measurements of $w(t).$ 
Instead, we consider algorithms that rely solely on the plant's 
measurable output $y(t),$ and are described by an internal state 
$z(t),$ evolving on an open subset $Z \subseteq \real^{n_c}$ for 
some $n_c \in \mathbb{N}_{\geq 0}.$
Based on this information, the algorithm generates a control input 
$u(t) \in \real^m$ at each time instant. 
See Fig.~\ref{fig:block_diagram_FO} for a block diagram illustration.
Formally, we consider the  following class of control 
algorithms\footnote{Although one could consider a 
more general control algorithm of the form $\dot z(t) = F_c(z(t), y(t))$ and 
$u(t) = G_c(z(t),y(t)),$ we will show in Section~\ref{sec:dynamic_feedback} 
that allowing for $G_c$ to depend on $y(t)$ is unnecessary. Hence, we 
focus the simpler formulation~\eqref{eq:controller_equatons} for the sake 
of notation.}:
\begin{align}\label{eq:controller_equatons}
\dot z(t) &= F_c(z(t), y(t)), \nonumber\\
u(t) &= G_c(z(t)),
\end{align}
where $\map{F_c}{Z \times \real^q}{\real^{n_c}},$ 
$\map{G_c}{Z}{\real^{m}},$ and the algorithm's state space dimension
$n_c \in \mathbb{N}_{\geq 0}$ are to be designed. 
We note that~\eqref{eq:controller_equatons} defines a general class 
of algorithms within which we will seek our design. 
We discuss in Remark~\ref{rem:basig_feedback_opt} how this class encompasses several existing methods as special cases.

Problem~\ref{prob:feedback_opt_informal} then involves designing the functions $F_c(z,y)$ and $G_c(z),$  together with $n_c,$ such 
that the {\it gradient error} signal:
\begin{align}\label{eq:gradient_signal}
    g(t) = \nabla_u \phi(u(t),w(t)),
\end{align}
satisfies $g(t) \rightarrow 0$  as $t \rightarrow\infty.$ 
Note that the gradient error signal cannot be evaluated by the 
controller, as $w(t)$ is unmeasurable. Instead, the controller must 
regulate $g(t)$ to zero despite only having access to the plant's 
output $y(t)$.

The dynamics of the plant~\eqref{eq:plant}, combined with the 
exosystem \eqref{eq:exosystem} and
controller~\eqref{eq:controller_equatons}, form a nonlinear 
autonomous system:
\begin{align}\label{eq:interconnected_system}
\dot x(t) &= f(x(t),u(t),w(t)),\nonumber 
& y(t) &= c(x(t),w(t)), \nonumber \\
\dot z(t) &= F_c(z(t), y(t)), \nonumber
& u(t) &= G_c(z(t)),\nonumber \\
\dot{w}(t) &=s(w(t)),
& g(t) &= \nabla_u \phi(u(t),w(t)).
\end{align}
See Fig.~\ref{fig:block_diagram_FO} for an illustration of the interconnection.

\begin{remark}[\textbf{\textit{Generality of the controller 
class~\eqref{eq:controller_equatons}}}]
\label{rem:basig_feedback_opt}
To solve problem~\eqref{eq:optimization_unconstrained}, the authors 
of~\cite{AH-SB-GH-FD:20,AH-ZH-SB-GH-FD:24}, under the additional assumption that 
$h(u,w)=\hat h(u)+w$, proposed the controller
\begin{align}\label{eq:feedback_opt_controller}
    \dot u(t) &= - \eta [\nabla_1\phi_0(u(t), y(t))
    + J^\tsp_{\hat h}(u(t)) \nabla_2\phi_0(u(t), y(t))],
\end{align}
where $\eta>0$ is a tunable gain, $J_{\hat h}$ denotes the Jacobian 
matrix of $\hat h(u)$, and $\nabla_1 \phi_0$, and $\nabla_2 \phi_0$ 
are used to denote the gradient of $\phi_0$ with respect to the first 
and second variable, respectively. When the disturbance $w(t)$ is 
constant and the plant~\eqref{eq:plant} is exponentially 
pre-stabilized, a sufficiently-small choice of the tunable gain 
$\eta$ ensures that condition~\ref{def:tracking_b} is 
met~\cite{AH-SB-GH-FD:20,MC-ED-AB:20}.
It is immediate to verify that~\eqref{eq:feedback_opt_controller} is a special 
case of our controller class~\eqref{eq:controller_equatons} with the choices:
\begin{align}\label{eq:feedback_opt_controller_functions}
F_c(z,y) &= - \eta [\nabla_1\phi_0(z, y)
    + J^\tsp_{\hat h}(u) \nabla_2\phi_0(u, y)],\notag\\
    G_c(z) &= z.
\end{align}
It follows that the controller structure~\eqref{eq:controller_equatons} is sufficiently general 
to include a number of existing methods as specific instances. Therefore, any 
conclusions drawn about the class of algorithms
in~\eqref{eq:controller_equatons} will also 
hold for these existing methods.
\QEDB\end{remark}

\subsection{Technical problem statement}

As stated previously, our goal is to design a control algorithm that 
(i) stabilizes the closed-loop system with zero disturbance locally 
about an equilibrium and (ii) regulates the state and control input 
to an optimal solution of~\eqref{eq:optimization} for all 
disturbances generated by the exosystem~\eqref{eq:exosystem}.
Therefore, we assume existence of an equilibrium point $(x_\circ^\star,z_\circ^\star,u_\circ^\star,y_\circ^\star)$ of the closed-loop system~\eqref{eq:interconnected_system} with zero disturbance for which the gradient is zero. That is\footnote{The equilibrium point can be constructed by first finding $x_\circ^\star$ and $u_\circ^\star$ such that $0 = f(x_\circ^\star,u_\circ^\star,0)$ and $0 = \nabla_u \phi(u_\circ^\star,0)$, and then constructing $y_\circ^\star$ and $z_\circ^\star$ along with the controller $F_c$ and $G_c$ to satisfy the remaining conditions.},
\begin{align}\label{eq:equilibrium}
    0 &= f(x_\circ^\star,u_\circ^\star,0),
    & 0 &= F_c(z_\circ^\star,y_\circ^\star), & 0 &= \nabla_u \phi(u_\circ^\star,0), \nonumber \\
    y_\circ^\star &= c(x_\circ^\star,0),
    & u_\circ^\star &= G_c(z_\circ^\star). &&
\end{align}
Moreover, we assume that $u_\circ^\star$ and $z_\circ^\star$ are locally unique.
We can now formalize the notion of 
exact asymptotic tracking.

\begin{definition}[\textit{\textbf{Exact asymptotic tracking}}]
\label{def:tracking}
The controlled plant~\eqref{eq:interconnected_system} is said to {\it exactly 
asymptotically track a critical trajectory} of~\eqref{eq:optimization} when 
the following conditions hold:
\begin{enumerate}[label={(D\ref{def:tracking}\textup{\alph*})},leftmargin=*]
\item \label{def:tracking_a}
The equilibrium $(x_\circ^*, z_\circ^\star)$ of the autonomous 
system:
\begin{align}\label{eq:autonomous_system_w=0}
\dot x(t) &= f(x(t),G_c(z(t)), 0), \nonumber \\
\dot z(t) &= F_c(z(t), c(x(t),G_c(z(t))) ), 
\end{align}
is locally exponentially stable.

\item \label{def:tracking_b}
There exists a neighborhood $\Upsilon \subseteq X \times Z \times W$ of 
$(x_\circ^\star, z_\circ^\star, 0)$ such that, for each initial condition 
$(x(0), z(0), w(0)) \in \Upsilon,$ the solution 
of~\eqref{eq:interconnected_system} satisfies:
\begin{align}
\lim_{t \to \infty} g(t)=0. \QEDBB
\end{align}
\end{enumerate}
\end{definition}

With this background, we are now ready to make the objectives of 
Problem~\ref{prob:feedback_opt_informal} mathematically rigorous.

\begin{problem}[\textbf{\textit{Feedback optimization algorithm design -- formal}}]
\label{prob:feedback_opt_formal}
Design, when possible, $F_c(z,y)$, $G_c(z)$, and $n_c$ so 
that~\eqref{eq:interconnected_system} exactly asymptotically tracks a 
critical trajectory of~\eqref{eq:optimization}.~
\QEDB\end{problem}

We conclude this section with an illustration of our framework.

\begin{example}[\textbf{\textit{State regulation for linear systems}}]
\label{ex:linear_quadratic}
Consider an instance of~\eqref{eq:plant} with linear dynamics:
\begin{align}\label{eq:linear_plant}
    \dot x(t) &= Ax(t) + Bu(t) + Pw(t), \nonumber\\
    y(t) &= C x(t) + Qw(t),
\end{align}
where $A \in \real^{n \times n}, B \in \real^{n \times m}, P \in \real^{n \times p}, C \in \real^{q \times n}, Q\in \real^{q \times  p}$. 
Consider an optimal state-regulation problem~\eqref{eq:optimization}, where the 
objective is to regulate the state of the plant to an optimal equilibrium that 
balances operational costs and control effort:
\begin{align}\label{eq:optimization_linear_regulation}
\underset{u \in \real^m, x \in \real^n}{\text{minimize}} ~~~& 
\tfrac{1}{2} \norm{x}^2 + \tfrac{\lambda}{2} \norm{u}^2,\nonumber\\
\text{subject to:} ~~ & 0  = Ax + Bu + Pw(t),
\end{align}
where $\lambda\geq 0$ is a regularization parameter. 
If $A$ is invertible, a mapping $h(u,w)$ satisfying 
Assumption~\ref{as:steadyStateMap} exists and is given by $h(u,w) = T_{xu} u + T_{xw} w$,
where $T_{xu} = - A^\inv B$ and $T_{xw} = - A^\inv P.$
The corresponding unconstrained problem~\eqref{eq:optimization_unconstrained} 
is:
\begin{align*} 
\underset{u \in \real^m}{\text{minimize}} ~~~& 
\tfrac{1}{2} \norm{T_{xu} u + T_{xw} w(t)}^2 + \tfrac{\lambda}{2} \norm{u}^2.
\end{align*}
Restricting, for simplicity of the illustration, the 
class~\eqref{eq:controller_equatons} to linear controllers, the objective of 
this work is thus to design: 
\begin{align}\label{eq:controller_linear}
    \dot z(t) &= A_c z(t) + B_c y(t), \nonumber\\
    u(t) &= C_c z(t), 
\end{align}
where $A_c \in \real^{n_c\times n_c}, B_c \in \real^{n_c\times q}, C_c \in \real^{m \times n_c},$ such that the interconnected system is exponentially 
stable and the gradient $g(t) = R u(t) + T w(t)$ converges to zero as $t\to\infty$, where $R \defeq T_{xu}^\tsp T_{xu} + \lambda I$ and $T \defeq T_{xu}^\tsp T_{xw}.$ \QEDB
\end{example}

\section{Necessary conditions for exact tracking and proposed controller structure}
\label{sec:necessary_conditions}

In this section, we derive a set of necessary conditions for the existence of an algorithm 
that solves Problem~\ref{prob:feedback_opt_formal}, and we outline the 
structure of the controller we propose to address this problem.

\subsection{Necessary conditions for exact asymptotic tracking}
To state our conditions, we will require the Jacobian matrices 
$A \in \real^{n \times n}, P \in \real^{n \times p}, C \in \real^{q \times n}, Q \in \real^{q \times p}, B \in \real^{n \times m}, S \in \real^{p \times p}$, and $T \in \real^{m\times p}$, defined as follows:
\begin{alignat}{2}\label{eq:jacobians}
A &\defeq \left[\frac{\partial f}{\partial x}\right]_{(x,u,w)=(x_\circ^\star, u_\circ^\star,0)},\nonumber & \quad P &\defeq \left[\frac{\partial f}{\partial w}\right]_{(x,u,w)=(x_\circ^\star, u_\circ^\star,0)},\\
C &\defeq \left[\frac{\partial c}{\partial x}\right]_{(x,w)=(x_\circ^\star,0)},\nonumber &
Q &\defeq \left[\frac{\partial c}{\partial w} \right]_{(x,w)=(x_\circ^\star,0)}, \nonumber\\
B &\defeq \left[\frac{\partial f}{\partial u}\right]_{(x,u,w)=(x_\circ^\star, u_\circ^\star,0)}, &
S &\defeq \left[\frac{\partial s}{\partial w}\right]_{w=0}, \nonumber\\
T &\defeq \left[\frac{\partial \nabla_u \phi }{\partial w}\right]_{(u,w)=(u_\circ^\star,0)}. &&
\end{alignat}
Recall also that we denote by $\mathscr{X}_{\geq}(A)$ the unstable subspace of 
$A$ and by $\mc O^c(C,A)$ the unobservable subspace of the pair $(C,A)$ (see 
Section~\ref{sec:intro}-Notation). 

\begin{proposition}[\textit{\textbf{Necessary conditions for exact tracking}}] 
\label{prop:necessary_conditions_tracking}
System~\eqref{eq:interconnected_system} exactly asymptotically track a critical 
trajectory of~\eqref{eq:optimization} only if the following conditions hold:
\begin{enumerate}
\item The pair $(A,B)$ is stabilizable.
\item The pair $(C,A)$ is detectable. 
\item The following inclusion holds:
\begin{align*}
\mathcal{O}^c(C_L, A_L) \cap \mathscr{X}_{\geq}(A_L) \subseteq \operatorname{Ker} 
\begin{bmatrix} 0_{m \times n} & T   \end{bmatrix}.
\end{align*}
where $\displaystyle A_L\defeq \bmat{A & P \\ 0 & S}$ and $\displaystyle C_L \defeq \bmat{C & Q}$.\QEDB
\end{enumerate}
\end{proposition}

\begin{proof}
The Jacobian matrix of the dynamics~\eqref{eq:autonomous_system_w=0} has the form: 
\begin{align}
\label{eq:aux_jacobian}
    \begin{bmatrix}
        A & B C_c \\
        B_c C &  A_c + B_c Q C_c
    \end{bmatrix},
\end{align}
where $A, B,$ and $C$ are as in~\eqref{eq:jacobians}, and 
\begin{align}\label{eq:jacobians_ctrl}
A_c &\defeq \left[\frac{\partial F_c}{\partial x}\right]_{(z,y)=(z_\circ^\star, y_\circ^\star)},\nonumber &  B_c &\defeq \left[\frac{\partial F_c}{\partial y}\right]_{(z,y)=(z_\circ^\star, y_\circ^\star)},\\
C_c &\defeq \left[\frac{\partial G_c}{\partial z}\right]_{z=z_\circ^\star}.
\end{align}
Since $(x_\circ^*, z_\circ^\star)$ is locally exponentially stable (by 
\ref{def:tracking_a}), all eigenvalues of the matrix 
in~\eqref{eq:aux_jacobian} must be in $\complex_{<}.$ A necessary condition for asymptotic tracking is then, for all $\lambda\in\complex_{\geq}$,
\begin{align*}
    \operatorname{Ker} \begin{bmatrix}
        A - \lambda I\\
        B_c C 
    \end{bmatrix} = 0
    \quad\text{which implies}\quad
    \operatorname{Ker} \begin{bmatrix}
        A - \lambda I\\
        C
    \end{bmatrix} = 0.
\end{align*}
This proves part 2) of the claim. 
Similarly, Hurwitz stability of the matrix in~\eqref{eq:aux_jacobian} 
necessitates that, for all $\lambda\in\complex_{\geq}$,
\begin{multline*}
    \operatorname{Im} \begin{bmatrix}
        A - \lambda I, & B C_c
    \end{bmatrix} = \real^n \\
    \quad\text{which implies}\quad
    \operatorname{Im} \begin{bmatrix}
        A - \lambda I& B
    \end{bmatrix} =\real^n,
\end{multline*}
thus proving part 1) of the claim. 
To prove 3) assume, by contradiction, that there exists a vector 
$v \in \mathcal{O}^c(C_L, A_L) \cap \mathscr{X}_{\geq}(A_L)$ that 
does not belong to  $\operatorname{Ker} \begin{bmatrix} 0 & T \end{bmatrix}$. 
This implies that there exists an unstable mode 
of~\eqref{eq:autonomous_system_w=0} in the direction $v$ for which the 
gradient signal~\eqref{eq:gradient_signal} is nonzero in a neighborhood of 
the origin of $W.$ But this contradicts~\ref{def:tracking_b} 
in Definition~\ref{def:tracking}.
\end{proof}

Condition 1) in the proposition asks that the plant 
is stabilizable; 2) that it is detectable; and 3) that any 
undetectable mode of the pair $(C_L, A_L)$ lies in the null space of 
the matrix $[0, T].$
Intuitively, the requirement 1) is needed to ensure that~\eqref{eq:plant} can 
be stabilized (see Definition~\ref{def:tracking_a}) and is a classical 
requirement in stabilization problems via state feedback; the 
requirement 2) is needed to ensure that the state of~\eqref{eq:plant} can be 
estimated from output measurements, and is a standard requirement in the 
existing literature on stabilization via output feedback; the requirement 3) 
is specific to our problem, and asks that all unstable modes 
of~\eqref{eq:exosystem} are either detectable from $y(t)$ or do 
not affect the gradient signal $g(t).$ Notice that condition 3) implies 
condition 2); moreover, condition 3) automatically holds when the pair 
$(C_L, A_L)$ is detectable. 
Motivated by these necessary conditions, in the remainder we 
will impose the following assumption.

\begin{assumption}[\textbf{\textit{Stabilizability and Detectability}}]
\label{as:stabilizability_detectability}
The pair $(A,B)$ is stabilizable and the pair $(C_L, A_L)$
is detectable. 
\QEDB\end{assumption}

Assumption~\ref{as:stabilizability_detectability} guarantees that the conditions of Proposition~\ref{prop:necessary_conditions_tracking} are met. 
Moreover, this assumption is in line with the feedback 
optimization literature; this assumption appears explicitly 
in~\cite{LL-ZN-EM-JS:18,LL-JS-EM:20} and 
implicitly in~\cite{AH-SB-GH-FD:20,GB-JC-JP-ED:21-tcns,LL-JS-EM:20}, where 
the plant is assumed to be pre-stabilized (so that 1) and 2) are implicitly 
required), and $w(t)$ is assumed bounded so that, together with the 
stability of the plant, the requirement 3) is ensured to hold.

\begin{figure}[tb]
\centering
\includegraphics[width=\columnwidth]{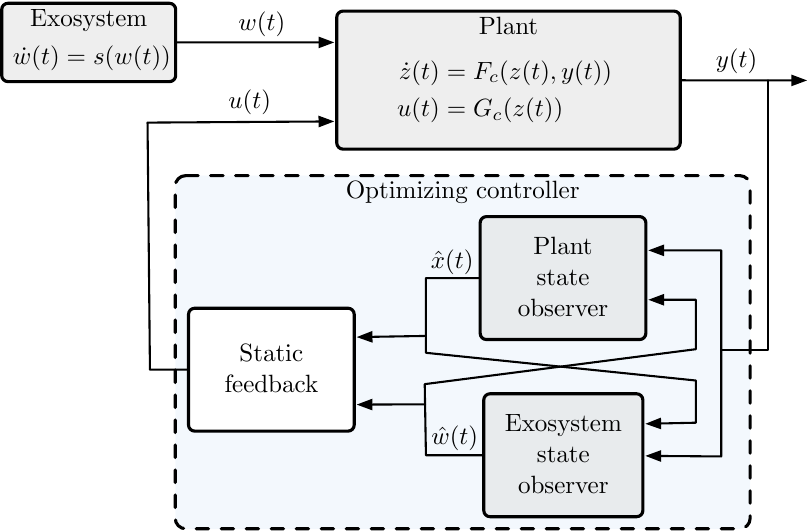}
\caption[]{Block-diagram representation of the optimizing controller 
developed in this work. 
The proposed controller is structured according to a separation principle, 
combining a state estimator -- responsible for reconstructing the unmeasurable 
states \((\hat{x}(t), \hat{w}(t))\) -- with a static feedback law that regulates 
the plant toward a point where the gradient error vanishes. See 
Section~\ref{sec:block_diagram_controller} and the controller 
components~\ref{cond:separation1}–\ref{cond:separation2}.
}
\label{fig:separation_principle_structure}
\vspace{-.5cm}
\end{figure}

\subsection{Structure of the proposed controller}
\label{sec:block_diagram_controller}
Proposition~\ref{prop:necessary_conditions_tracking} suggests the existence of 
a separation principle~\cite{AA-HK:99} between the problems of asymptotic 
stabilization 
(see Definition~\ref{def:tracking_a}) and gradient error regulation 
(see Definition~\ref{def:tracking_b}). 
Concretely, this is reflected in the need for two independent properties: 
stabilizability and detectability.
Although a separation is, for now, only necessary, our control design 
technique relies on showing that such a separation is also sufficient to 
address Problem~\ref{prob:feedback_opt_formal}. Establishing this fact will be 
the objective of the subsequent sections. 
Motivated by this observation, we anticipate that the controller we propose in 
this work has a structure that is inspired from such a separation principle.
Specifically, it consists of two main components (C):
\begin{enumerate}[label=(C\arabic*), leftmargin=*]
\item  \label{cond:separation1}
A static feedback algorithm (based on the dynamic states $x(t), w(t)$, or 
their estimates \(\hat{x}(t), \hat{w}(t)\)), responsible for stabilizing the 
closed-loop system (i.e., ensure~\ref{def:tracking_a}) and drive the plant 
to operating conditions where the gradient error signal \(g(t)\) vanishes 
(i.e., achieve~\ref{def:tracking_b}).
\item  \label{cond:separation2}
A dynamic observer, responsible for estimating the unmeasurable dynamic 
states $x(t)$ and $w(t)$ (i.e., generating the estimates $\hat x(t)$ and 
$\hat w(t)$).
\end{enumerate}

A block diagram of the proposed controller is illustrated in 
Fig.~\ref{fig:separation_principle_structure}.
Developing the components~\ref{cond:separation1}--\ref{cond:separation2} is 
the objective of the subsequent sections; precisely, 
\ref{cond:separation1} is developed in Section~\ref{sec:static_feedback}, 
the effectiveness of this controller structure is discussed in 
Section~\ref{sec:dynamic_feedback}, and  
\ref{cond:separation2} is developed in Section~\ref{sec:alg_design}.


\section{The static-feedback optimization problem}
\label{sec:static_feedback}
In this section, we develop the controller component~\ref{cond:separation1}. 
Precisely, we begin by considering an idealization of the problem where the 
disturbance $w(t)$ and the plant's internal state $x(t)$ can be 
measured by the control algorithm. 
In these cases, because the control algorithm has access to all the dynamic 
states in~\eqref{eq:interconnected_system}, the controller 
model~\eqref{eq:controller_equatons} can be replaced by a static 
(i.e., algebraic) map of the form\footnote{While one could consider a dynamic 
control algorithm of the form $\dot z(t) = F_c(z(t), w(t),x(t))$ and 
$u(t) = G_c(z(t)),$ we will prove in 
Theorem~\ref{thm:existence_parameter_feedback} that such a dynamic structure 
is unnecessary.}:
\begin{align}\label{eq:parameter_feedback}
u(t) &= H_c(x(t),w(t)),
\end{align}
where $\map{H_c}{X \times W}{\real^m}$ is a $C^0$ mapping to be designed. 
In line with the requirements~\eqref{eq:equilibrium}, we will assume that the 
mapping $H_c(x,w)$ to be designed is such that
$H_c(x_\circ^\star,0)=u_\circ^\star.$

In other words, \eqref{eq:parameter_feedback} assumes that the state of the 
plant $x(t)$ and the exogenous disturbance $w(t)$ are measurable and can be 
used directly for feedback. Because of the explicit dependence on $x(t)$ and 
$w(t),$ we will refer to~\eqref{eq:parameter_feedback} to as a 
\textit{static-feedback optimization} algorithm. 


the framework  developed here will be used in combination with the controller 
component~\ref{cond:separation1}  in Section~\ref{sec:dynamic_feedback}  to 
tackle Problem~\ref{prob:feedback_opt_formal} in its full generality.

Composing~\eqref{eq:plant}, \eqref{eq:exosystem}, and 
\eqref{eq:parameter_feedback} yields the closed-loop system:
\begin{align}
\label{eq:copled_system_parameter_feedback}
\dot x(t) &= f(x(t),u(t),w(t)),\nonumber 
& y(t) &= c(x(t),w(t)), \nonumber \\
\dot{w}(t) &=s(w(t)), \nonumber\\
u(t) &= H_c(x(t),w(t)), 
& g(t) &= \nabla_u \phi(u(t),w(t)).
\end{align}

With this framework, Problem~\ref{prob:feedback_opt_formal} is reformulated 
as follows.

\begin{problem}[\textbf{\textit{Static exact asymptotic tracking}}]
\label{prob:parameter_feedback}
Find, if possible, $H_c(x,w)$ such that:
\begin{enumerate}[label={(P\ref{prob:parameter_feedback}\textup{\alph*})},leftmargin=*]
\item \label{prob:parameter_feedback_a}
The equilibrium $x_\circ^*$ of the autonomous system:
\begin{equation}\label{eq:autonomous_staticfeedback}
\dot x(t) = f(x(t),H_c(x(t),0),0),
\end{equation}
is locally exponentially stable.

\item \label{prob:parameter_feedback_b}
There exists a neighborhood $\Upsilon \subseteq X \times W$ of 
$(x_\circ^\star, 0)$ such that, for each initial condition 
$(x(0), w(0)) \in \Upsilon,$ the solution 
of~\eqref{eq:copled_system_parameter_feedback} satisfies:
\begin{align}\label{eq:tracking_g_parameterfeedback}
\lim_{t \to \infty} g(t)=0.
\end{align}
\end{enumerate}
When these conditions hold, \eqref{eq:copled_system_parameter_feedback} is 
said to {\it exactly asymptotically track a critical trajectory} of~\eqref{eq:optimization}.
\QEDB\end{problem}

The following definition is instrumental to our characterization of this problem.

\begin{definition}[\textbf{\textit{Limit point and limit set}}]
A point $\bar w \in W$ is a \emph{limit point with respect 
to the initialization $w_\circ \in W$} if there exists a non-decreasing 
sequence $\{k_i\}_{i\in\naturalnneg},$ with 
${k_i\to\infty}$ as ${i\to\infty},$ such that the trajectory
of~\eqref{eq:exosystem} with $w(0) = w_\circ$ 
satisfies $w({k_i})\to \bar w$ as ${i\to\infty}$. 
For $w_\circ \in W,$ let $\Omega(w_\circ)$ denote the set of all limit points 
(i.e., for all sequences $\{k_i\}_{i\in\naturalnneg}$) 
of~\eqref{eq:exosystem} with respect to the initialization 
$w_\circ.$
Given $W_\circ \subseteq W,$ the set $\Omega(W_\circ) \defeq 
\cup_{w_\circ \in W_\circ} \Omega(w_\circ)
$ 
is called the \textit{limit set with respect 
to initializations in $W_\circ$}~\cite{GDB:27}. 
\QEDB\end{definition}

Intuitively, $\Omega(W_\circ)$ denotes the set of all limit 
points (equilibria, limit cycles, etc.) that can be reached by the disturbance 
state when initialized at points in $W_\circ$. 
By Assumption~\ref{as:exosystem}, 
$\Omega(W_\circ)$ is contained in some neighborhood of the origin 
of $\real^p$, whose radius depends on the initialization set $W_\circ.$

The next result provides necessary and sufficient conditions for the 
existence of a static-feedback optimization algorithm.

\begin{theorem}[\textbf{\textit{Solvability of the parameter-feedback tracking problem}}]
\label{thm:existence_parameter_feedback}
Let Assumptions~\ref{as:lipschitz_convexity}--\ref{as:stabilizability_detectability} 
hold. Problem~\ref{prob:parameter_feedback} is solvable if and only if there 
exist $C^2$ mappings 
$\pi: W_\circ \to X$ and $\gamma: W_\circ \to \real^m,$
where $W_\circ\subset W$ is some neighborhood of the origin of $\real^p$, such 
that:
\begin{subequations}
\label{eq:par_feedback_existence}
\begin{align}
\frac{\partial \pi}{\partial w} s(w) &= f(\pi(w),\gamma(w), w),\label{eq:par_feedback_existence_a}\\
0 &= \nabla_u \phi(\gamma(w),w), \label{eq:par_feedback_existence_b}
\end{align}
\end{subequations}
hold at all limit points $w \in \Omega(W_\circ)$.~
\QEDB\end{theorem}

The proof of this result builds on 
Theorem~\ref{thm:characterization_parameter_feedback} (presented shortly 
below), and hence is postponed to the appendix.

Theorem~\ref{thm:existence_parameter_feedback} asserts that the 
solvability of the static-feedback optimization problem depends upon the 
existence of two mappings $x=\pi(w)$ and $u=\gamma(w)$ that make the gradient 
of the cost identically zero (cf.~\eqref{eq:par_feedback_existence_b}) and that 
relate the plant dynamics $f(x,u,w)$ with the exosystem $s(w)$  in a 
neighborhood of each limit point of the disturbance 
(cf.~\eqref{eq:par_feedback_existence_a}). Interestingly, the solvability 
of~\eqref{eq:par_feedback_existence} can be related to the existence of zero 
dynamics~\cite{AI:13} for a composite system that incorporates the plant and 
the exosystem, as discussed in the following remark.

\begin{remark}[\textbf{\textit{Interpretation of~\eqref{eq:par_feedback_existence} in 
terms of zero dynamics}}]
\label{rem:zero_dynamics}
Conditions~\eqref{eq:par_feedback_existence} state that an algorithm that 
solves Problem~\ref{prob:parameter_feedback} exists if and only if there 
exists a submanifold of the state space (i.e., $M_s = \{(x,w): x = \pi(w)\}$) 
such that:
\begin{enumerate}[(i)]
\item for some choice of feedback law $u(t) = \gamma(w(t))$, the trajectories of the closed-loop 
system~\eqref{eq:copled_system_parameter_feedback} starting in this manifold 
remain in this manifold (cf.~\eqref{eq:par_feedback_existence_a}) 
\item  the corresponding gradient is identically zero 
(cf.~\eqref{eq:par_feedback_existence_b})
\item the flow of the zero dynamics on this invariant manifold is a 
diffeomorphic copy of the disturbance flow
(cf.~\eqref{eq:par_feedback_existence_a})
\end{enumerate}
It follows from (i)--(ii) that $M_s$ must be contained in the zero dynamics 
manifold of the composite system that incorporates the plant and the 
exosystem; in addition, (iii) requires that the flow of the zero dynamics of 
the composite system on $M_s$ must be a diffeomorphic copy of the flow of the exosystem.~
\QEDB\end{remark}

While Theorem~\ref{thm:existence_parameter_feedback} provides a set of 
conditions for the solvability of the static exact asymptotic tracking problem, 
the question of how to design such an algorithm remains unanswered. 
As an intermediate step to address this question, we present the following 
result, which characterizes the class of all static-feedback algorithms 
that achieve asymptotic tracking.

\begin{theorem}[\textbf{\textit{Characterization of static-feedback optimization algorithms}}]
\label{thm:characterization_parameter_feedback}
%
%
Let Assumptions~\ref{as:lipschitz_convexity}--\ref{as:stabilizability_detectability}
hold, and assume that $H_c(x,w)$ is such that condition 
\ref{prob:parameter_feedback_a} is met. 
Then, \ref{prob:parameter_feedback_b} holds if and only if there 
exists a $C^2$ mapping $\pi: W_\circ \to X,$ with $W_\circ\subset W$ some 
neighborhood of the origin of $\real^p$, such  that:
\begin{subequations}
\label{eq:par_feedback}
\begin{align}
\frac{\partial \pi}{\partial w} s(w) &= f(\pi(w),H_c(\pi(w),w),w),\label{eq:par_feedback_a}\\
0 &= \nabla_u \phi(H_c(\pi(w),w),w), \label{eq:par_feedback_b}
\end{align}
\end{subequations}
hold at all limit points $w \in \Omega(W_\circ)$.~
\QEDB\end{theorem}

\begin{proof}
{\it (Only if)} Suppose $g(t) \rightarrow 0$ as $t \rightarrow \infty$; we will 
show that~\eqref{eq:par_feedback} holds. 
The closed-loop system~\eqref{eq:interconnected_system} has the form:
\begin{align}\label{eq:closed_loop_static_2}
\dot x &= (A+BK) x + (P+BL) w + \phi(x,w), \notag\\
\dot w &= Sw + \psi(w),
\end{align}
where $A, B, P, S$ are defined in \eqref{eq:jacobians}, 
$K \defeq \left[\frac{\partial H_c }{\partial x}\right]_{(x,w)=(x_\circ^\star,0)},$
$L \defeq \left[\frac{\partial H_c }{\partial w}\right]_{(x,w)=(x_\circ^\star,0)},$
and $\phi(x,w)$ and $\psi(w)$ are functions that vanish at 
$(x,w)=(x_\circ^\star,0),$ together with their first-order derivatives. 
By assumption, the eigenvalues of $(A+B K)$ are in $\complex_{<}$, and those 
of $S$ are on the imaginary axis. By 
Theorem~\ref{thm:existence_center_manifold}, the system admits a center 
manifold at $(x_\circ^\star,0)$, which can be represented as the graph of a
continuous mapping  $x = \pi(w),$ with $\pi(w)$
satisfying~\eqref{eq:par_feedback_a}. This 
establishes~\eqref{eq:par_feedback_a}. 

To establish~\eqref{eq:par_feedback_b}, note that, under 
Assumption~\ref{as:exosystem}, there exists a neighborhood $W_\circ$ of the origin 
such that every trajectory of~\eqref{eq:exosystem} initialized in $W_\circ$ remains 
bounded. Consequently, each such trajectory admits a subsequence that converges 
to a limit point $\bar w \in \Omega(W_\circ).$
Furthermore, we have:
\begin{align*}
\lim _{i \rightarrow \infty} g({k_i})
&=\lim _{i \rightarrow \infty} \nabla_u \phi\left(H_c\left(\pi(w(k_i)),w(k_i)\right), w(k_i)\right)\\
&=\nabla_u \phi\left(H_c\left(\pi(\bar w),\bar w\right), \bar w\right).
\end{align*}
When $\lim _{i \rightarrow \infty} g({k_i}) =0,$ the left-hand side is zero, 
which implies that~\eqref{eq:par_feedback_b} holds at $\bar w.$
Since this must hold at every limit point $\bar w$ and by {continuity of 
$H_c(\cdot, \cdot),$ \eqref{eq:par_feedback_b} must hold everywhere in 
a neighborhood of each point of $\Omega(W_\circ)$.}

{\it (Only if)} We now prove that \eqref{eq:par_feedback} implies 
$\lim_{t \to \infty} g(t)=0.$ By Theorem~\ref{thm:local_attractivity_manifold}, 
the center manifold $x =\pi(w)$ is locally 
attractive; namely, $x(t) \to \pi(w(t))$ as $t\to \infty$. Then, the fulfillment
of~\eqref{eq:par_feedback_b} guarantees that $g(t) \to 0$.
\end{proof}

Theorem~\ref{thm:characterization_parameter_feedback} provides a complete 
characterization of the class of static-feedback optimization algorithms that 
achieve exact asymptotic tracking. 
In words, a control algorithm~\eqref{eq:parameter_feedback} 
satisfying~\ref{prob:parameter_feedback_a} asymptotically tracks a critical 
trajectory if and only if, on a neighborhood of the 
limit set of the exosystem  $\Omega(W_\circ)$, the composite function 
$H_c(\cdot,w) \circ \pi(\cdot)$ zeros the gradient
(cf.~\eqref{eq:par_feedback_b}) and the plant $f(x,u,w)$ is algebraically 
related to the disturbance $s(w)$ as described by~\eqref{eq:par_feedback_a}.
Finally, it is worth commenting on the assumption
that $H_c(x,w)$ satisfies~\ref{prob:parameter_feedback_a} in the statement of 
Theorem~\ref{thm:characterization_parameter_feedback}: conditions for the 
existence of such mappings are given in 
Theorem~\ref{thm:existence_parameter_feedback} and a technique to design such 
$H_c(x,w)$ is given shortly below (see~\eqref{eq:expression_Hc}).
In this context, the value of 
Theorem~\ref{thm:characterization_parameter_feedback} lies in the 
conditions~\eqref{eq:par_feedback}, which will be used below to 
design $H_c(x,w)$ in place of~\ref{prob:parameter_feedback_b}.

The proofs of Theorems~\ref{thm:existence_parameter_feedback}--\ref{thm:characterization_parameter_feedback} are constructive, 
as they provide a method to design control algorithms that solve the 
asymptotic tracking problem (Problem~\ref{prob:parameter_feedback}). 
Explicitly, given mappings $x=\pi(w)$ and $u = \gamma(w)$ satisfying 
\eqref{eq:par_feedback_existence}, a parameter feedback algorithm that solves 
Problem~\ref{prob:parameter_feedback} is given by:
\begin{align}\label{eq:expression_Hc}
    H_c(x,w) = \gamma(w) + K(x - \pi(w)),
\end{align}
where $K$ is any matrix such that the eigenvalues of $A+BK$ are in 
$\complex_{<}.$ 
It is also worth noting that the control action in~\eqref{eq:expression_Hc} is 
the superposition of two terms: a state-error action $K(x - \pi(w)),$ 
responsible for stabilizing the plant around the manifold $\pi(w)$, and a 
control action $\gamma(w),$ responsible for zeroing the gradient. 
We illustrate the design procedure next.

\begin{example}[\textbf{\textit{Illustration of the design procedure for static-feedback optimization algorithms}}]
\label{ex:linear_quadratic_pt2}
Consider the control problem discussed in Example~\ref{ex:linear_quadratic}.
By application of Theorem~\ref{thm:existence_parameter_feedback}, there exists 
a state-feedback optimization algorithm achieving exact asymptotic tracking
if and only if there exist matrices 
$\Pi \in \real^{n \times p}, \Gamma \in \real^{m \times p}$ such that the 
following identities hold:
\begin{align}\label{eq:linear_static_feedback_solvability}
\Pi S &= A \Pi + B \Gamma +  P,\nonumber\\
0 &= R \Gamma + T.
\end{align}
Note that, in this case, the dependence on $w$ can be dropped, since these 
identities must hold anywhere in a neighborhood of the origin of $\real^p.$
When these two conditions hold, an algorithm solving 
Problem~\ref{prob:parameter_feedback} can be 
computed from~\eqref{eq:expression_Hc}, yielding
\begin{align*}
u(t) &= \Gamma w(t) + K(x(t) - \Pi w(t)) ,
\end{align*}
where $K$ is any matrix such that $A+BK$ is Hurwitz. 
\QEDB\end{example}

\begin{remark}[\textbf{\textit{Knowledge of the limit set}}]
In applications, the limit set 
$\Omega(W_\circ)$ may be unknown. When this is the case, 
to design a static-feedback optimization algorithm, it is possible to seek 
mappings $\gamma(w)$ and $\pi(w)$ that verify~\eqref{eq:par_feedback} on some 
set that includes $\Omega(W_\circ)$, which can be more easily 
determined (e.g., when $w(t)$ is periodic or uniformly bounded).
This approach overcomes the need 
for knowing the limit set precisely.~
\QEDB\end{remark}

\section{The dynamic feedback-optimization problem: existence and conditions for asymptotic tracking}
\label{sec:dynamic_feedback}

In this section, we formalize the effectiveness of a controller architecture 
based on the two independent components
\ref{cond:separation1}--\ref{cond:separation2}.
Precisely, we now show how the conclusions drawn in 
Section~\ref{sec:static_feedback} extend when the dynamic 
controller~\eqref{eq:controller_equatons} is used in place of the static 
one~\eqref{eq:parameter_feedback}.
We begin with the following solvability result.

\begin{theorem}[\textbf{\textit{Solvability of the dynamic feedback-optimization problem}}]
\label{thm:existence_dynamic_feedback}
Let Assumptions~\ref{as:lipschitz_convexity}--\ref{as:stabilizability_detectability} 
hold. Problem~\ref{prob:feedback_opt_formal} is solvable if and only if there 
exist $C^2$ mappings $\pi: W_\circ \to X$ and $\gamma: W_\circ \to \real^m,$ 
where $W_\circ\subset W$ is some neighborhood of the origin of $\real^p,$ such 
that:
\begin{subequations}
\label{eq:dyn_feedback_existence}
\begin{align}
\frac{\partial \pi}{\partial w} s(w) &= f(\pi(w),\gamma(w), w), \label{eq:dyn_feedback_existence_a}\\
0 &= \nabla_u \phi(\gamma(w),w), \label{eq:dyn_feedback_existence_b}
\end{align}
\end{subequations}
hold at all limit points $w \in \Omega(W_\circ)$.~
\QEDB\end{theorem}

The proof of this result builds on 
Theorem~\ref{thm:characterization_dynamic_feedback} (presented shortly 
below); hence, we postpone it to the appendix.

Interestingly, the conditions for the solvability of the static problem 
(Problem~\ref{prob:parameter_feedback}) derived in 
Theorem~\ref{thm:existence_parameter_feedback} and those for the solvability 
of the dynamic problem (Problem~\ref{prob:feedback_opt_formal}) derived in 
Theorem~\ref{thm:existence_dynamic_feedback} are identical. 
This shall not be surprising since, on the one hand, the static optimization 
algorithm~\eqref{eq:parameter_feedback} has access to measurements of the 
signals $x(t)$ and $w(t)$ while, on the other hand, the dynamic optimization 
algorithm~\eqref{eq:controller_equatons} has additional flexibility thanks to 
the availability of a dynamic state $z(t)$. We will reinterpret this property 
under the lens of additional results shortly below 
(see~\eqref{eq:change_of_coordinates}). 

The following result provides a characterization of the class of all 
dynamic-feedback optimization algorithms that achieve exact asymptotic 
tracking.

\begin{theorem}[\textbf{\textit{Characterization of dynamic feedback-optimization algorithms}}]
\label{thm:characterization_dynamic_feedback}
Let Assumptions~\ref{as:lipschitz_convexity}--\ref{as:stabilizability_detectability}
hold, and assume that the controller~\eqref{eq:controller_equatons} is such 
that condition~\ref{def:tracking_a} is met. Then, \ref{def:tracking_b}
holds if and only if there exist $C^2$ mappings 
$\pi: W_\circ \to X$ and $\sigma: W_\circ \to Z,$ with 
$W_\circ\subset W$ some neighborhood of the origin of $\real^p$, such  that:
\begin{subequations}\label{eq:dyn_feedback}
\begin{align}
\frac{\partial \pi}{\partial w} s(w) &= f(\pi(w),G_c(\sigma(w)), w), \label{eq:dyn_feedback_a}\\
\frac{\partial \sigma }{\partial w} s(w) &= F_c(\sigma(w),c(\pi(w),w)), \label{eq:dyn_feedback_b}\\
0 &= \nabla_u \phi(G_c(\sigma(w)),w), \label{eq:dyn_feedback_c}
\end{align}
\end{subequations}
hold at all limit points $w \in \Omega(W_\circ)$.~
\QEDB\end{theorem}

\begin{proof}
{\it (Only if)} Suppose that 
conditions~\ref{def:tracking_a}–\ref{def:tracking_b} are satisfied. We will now 
show that~\eqref{eq:dyn_feedback} holds.
The closed-loop system~\eqref{eq:copled_system_parameter_feedback} has the form:
\begin{align}\label{eq:closed_loop_dynamic_2}
\dot x &=  Ax + B C_c z + P w + \phi(x,z,w), \notag\\
\dot z &=  B_c C x + A_c z + B_cQ w + \chi(x,z,w), \notag\\
\dot w &= Sw + \psi(w),
\end{align}
where all matrices involved are defined in \eqref{eq:jacobians}, 
\eqref{eq:jacobians_ctrl}, and $\phi(x,z,w)$, $\chi(x,z,w),$ and $\psi(w)$ are 
functions that vanish at $(x,z,w)=(x_\circ^\star,z_\circ^\star,0),$ together 
with their first-order derivatives. 
By~\ref{def:tracking_a}, the eigenvalues of the matrix
\begin{align*}
    \begin{bmatrix}
        A & BC_c\\ B_c C & A_c
    \end{bmatrix}
\end{align*}
are in $\complex_{<}$, and those of $S$ are on the imaginary axis. 
By Theorem~\ref{thm:existence_center_manifold}, the system admits a center 
manifold at $(x_\circ^\star,z_\circ^\star,0)$, which can be expressed as the 
graph of two continuous mappings $x = \pi(w)$ and  $z=\sigma(w)$ that
satisfy~\eqref{eq:dyn_feedback_a} and~\eqref{eq:dyn_feedback_b}. 

To establish~\eqref{eq:dyn_feedback_c}, note that, under 
Assumption~\ref{as:exosystem}, there exists a neighborhood $W_\circ$ of the origin 
such that every trajectory of~\eqref{eq:exosystem} starting in $W_\circ$ is 
bounded, and therefore it has a subsequence that converges to some limit point 
$\bar w \in \Omega(W_\circ)$. 
Thus, we have:
\begin{align*}
\lim _{i \rightarrow \infty} g({k_i})
&=\lim _{i \rightarrow \infty} \nabla_u \phi\left(G_c\left(\sigma(w(k_i))\right), w(k_i)\right),\\
&=\nabla_u \phi\left(G_c\left(\pi(\bar w)\right), \bar w\right).
\end{align*}
When $\lim _{i \rightarrow \infty} g({k_i}) =0,$ the left-hand side is zero, 
which implies that~\eqref{eq:dyn_feedback_c} holds at $\bar w.$
Since this must hold at every limit point $\bar w$ and by continuity of 
$G_c(\cdot),$ \eqref{eq:dyn_feedback_c} must hold everywhere in 
a neighborhood of each point of $\Omega(W_\circ)$. This 
establishes~\eqref{eq:dyn_feedback_c}.

{\it (If)} We now show that, if condition~\ref{def:tracking_a} is 
satisfied and~\eqref{eq:dyn_feedback} holds, then~\ref{def:tracking_b} holds. 
By Theorem~\ref{thm:local_attractivity_manifold}, the center manifold 
$x =\pi(w)$ and $z=\sigma(w)$ are locally attractive; namely, $x(t) \to \pi(w)$ 
and $z(t) \to \sigma(w)$ as $t\to \infty$. Then, fulfillment
of~\eqref{eq:dyn_feedback_b} guarantees that $g(t) \to 0$, thus 
establishing~\ref{def:tracking_b}.
\end{proof}

The conditions in~\eqref{eq:dyn_feedback} fully characterize the class of 
dynamic-feedback algorithms that achieve asymptotic tracking.
In other words, an algorithm from the class~\eqref{eq:interconnected_system} 
satisfying~\ref{def:tracking_a} achieves exact tracking if 
and only if, on a neighborhood of the limit set of the exosystem  
$\Omega(W_\circ)$, the composite function $G_c \circ \sigma$ zeros the 
gradient (cf.~\eqref{eq:dyn_feedback_c}) and both the plant and the 
controller are algebraically related to the exosystem as given by, 
respectively, \eqref{eq:dyn_feedback_a} and \eqref{eq:dyn_feedback_b}.
It is also worth relating 
Theorem~\ref{thm:characterization_dynamic_feedback} with 
Theorem~\ref{thm:characterization_parameter_feedback}: by comparison, it is 
immediate to see that the design of a dynamic algorithm must follow the same 
criteria of the design of a static algorithm (compare 
\eqref{eq:dyn_feedback_a}, \eqref{eq:dyn_feedback_c} with 
\eqref{eq:par_feedback}), now with the additional requirement that 
the controller is algebraically related to the exosystem, as given 
by~\eqref{eq:dyn_feedback_c}. This property establishes the correctness of the 
controller architecture proposed 
in~\ref{cond:separation1}--\ref{cond:separation2}.

It is worth commenting on the requirement 
that~\eqref{eq:controller_equatons} satisfies~\ref{def:tracking_a} in the 
statement: conditions for the existence of an algorithm 
satisfying~\ref{def:tracking_a} are given in 
Theorem~\ref{thm:existence_dynamic_feedback} and a technique to design a 
stabilizing controller is given shortly below (see 
Algorithm~\ref{alg:dynamic_feedback_design}). In analogy with 
Theorem~\ref{thm:characterization_parameter_feedback}, the value of
Theorem~\ref{thm:characterization_dynamic_feedback} lies in the 
conditions~\eqref{eq:dyn_feedback}, which will be used in 
Algorithm~\ref{alg:dynamic_feedback_design} to design $F_c(z,y)$ and $G_c(z)$ 
in place of~\ref{def:tracking_b}.


We illustrate these findings in the following example.




\begin{example}[\textbf{\textit{Illustration of the conditions for 
dynamic feedback-optimization algorithms}}]
\label{ex:linear_quadratic_pt3}
Consider the problem discussed in 
Examples~\ref{ex:linear_quadratic}--\ref{ex:linear_quadratic_pt2}. 
By application of Theorem~\ref{thm:existence_dynamic_feedback}, 
Problem~\ref{prob:feedback_opt_formal} is solvable if and only there exist 
matrices $\Pi \in \real^{n \times p}$ and $\Gamma \in \real^{m \times p}$ such 
that~\eqref{eq:linear_static_feedback_solvability} hold. 
Notice that, as observed immediately after 
Theorem~\ref{thm:existence_dynamic_feedback}, the conditions for the 
solvability of the dynamic problem (Problem~\ref{prob:feedback_opt_formal})
coincide with the conditions for the solvability of the 
static problem (Problem~\ref{prob:parameter_feedback}).

Suppose now that~\eqref{eq:linear_static_feedback_solvability} holds,  
consider the controller~\eqref{eq:controller_linear}, and assume that  
$A_c$ and $B_c$ are such that the origin of:
\begin{align*}
    \dot x(t) &= A x(t) + B C_c z(t), \\
    \dot z(t) &= A_c z(t) + B_c C x(t),
\end{align*}
is exponentially stable. 
By application of Theorem~\ref{thm:characterization_dynamic_feedback}, 
exact asymptotic tracking is achieved if and only if there exist matrices 
$\Pi \in \real^{n \times p}$ and $\Sigma \in \real^{n_c \times p}$ such that:
\begin{align}\label{eq:design_equations_linear}
\Pi S &= A \Pi + B C_c \Sigma + P,\nonumber\\
\Sigma S &= A_c \Sigma +B_c(C \Pi + Q),\nonumber\\
0&= R C_c + T  \Sigma.
\end{align}
Given $A_c, B_c, C_c,$  \eqref{eq:design_equations_linear} is a 
set of linear equations in the unknowns $\Pi$ and $\Sigma,$ which can thus be 
checked efficiently.~
\QEDB\end{example}

We conclude this section by discussing an important implication stemming from 
Theorem~\ref{thm:characterization_dynamic_feedback}.
By~\eqref{eq:dyn_feedback_b}, the state of the control algorithm $z$ and that 
of the exosystem $w$ must be related, everywhere in the limit set $\Omega(W_\circ),$ by the 
relationship:
\begin{align}\label{eq:change_of_coordinates}
z =\sigma(w).
\end{align}
Condition~\eqref{eq:change_of_coordinates} can be interpreted as the existence 
of a coordinate transformation between the state of the exosystem and that of 
the control algorithm. This, in turn, implies that any controller capable of 
achieving exact tracking must use an internal state $z(t)$ that is a 
reduplicated copy of the disturbance signal $w(t),$ albeit potentially 
expressed in a different coordinate system.
We discuss further this property in 
Remarks~\ref{rem:internal_model_principle}--\ref{rem:internal_model_interpret_feedopt}, and illustrate some of its 
properties in Example~\ref{ex:injective_sigma}.

\begin{remark}[\textbf{\textit{The internal model principle of feedback optimization}}]
\label{rem:internal_model_principle}
We interpret condition~\eqref{eq:dyn_feedback} as the \textit{internal model 
principle of feedback optimization,} akin to its counterpart in control 
systems~\cite{ED:76,BF-WW:76,BF:77,AI-CIB:90} and time-varying 
optimization~\cite{GB-BVS:24-arxiv}. 
In fact, \eqref{eq:dyn_feedback} expresses the requirement that any feedback 
optimization algorithm that achieves asymptotic tracking must include an 
internal model of the exosystem.
The use of a copy of the temporal variability of the optimization 
problem is explicit in existing feedback optimization 
algorithms~\cite{GB-JC-JP-ED:21-tcns,AH-SB-GH-FD:20}, where an internal model 
for an integrator-type exosystem is used to reject constant disturbances. In  
this sense, our characterization encompasses all these existing approaches as 
special cases.
\QEDB\end{remark}

\begin{remark}[\textbf{\textit{Internal-model based interpretation of existing 
feedback optimization algorithms}}]
\label{rem:internal_model_interpret_feedopt}
Recall that the basic feedback-optimization 
algorithm~\eqref{eq:feedback_opt_controller} can be 
viewed as an instance of~\eqref{eq:controller_equatons} with 
$F_c(z,y)$ and $G_c(z)$ as in~\eqref{eq:feedback_opt_controller_functions} (see 
Remark~\ref{rem:basig_feedback_opt}).
By direct substitution into~\eqref{eq:dyn_feedback}, it is immediate to 
see that this algorithm 
satisfies~\eqref{eq:dyn_feedback} with $s(w) = 0$ and $\sigma(w)$ 
arbitrary. 
Since $s(w) = 0$ is the internal model of a constant signal, it 
follows from Theorem~\ref{thm:characterization_dynamic_feedback} that these 
algorithms can achieve exact asymptotic tracking only when $w(t)$ is a 
constant signal. This observation is in line with the inexact convergence 
results given for these algorithms when $w(t)$ is 
time-varying~\cite{GB-JC-JP-ED:21-tcns,GB-MV-JC-ED:24-tac}, and motivates the 
need for developing new algorithms when $w(t)$ is time-varying.
\QEDB\end{remark}

\begin{example}[\textbf{\textit{The mapping $\sigma(w)$ may not be injective}}]
\label{ex:injective_sigma}
Consider the scalar plant: $\dot x =  x + u$ with output $y = -2x + w$
subject to a constant disturbance ${\dot w = 0}$. Denoting $w\defeq w(t)$ for compactness, consider the following 
instance of~\eqref{eq:optimization_unconstrained},
\begin{align*}
\underset{u \in \real}{\text{minimize}} ~~~& \tfrac{1}{2} \vert u \vert^2 + 
\tfrac{1}{2} \vert w \vert^2,
\end{align*}
which results in the gradient signal: $g(t) = u(t) + w$.
In this case, a possible choice of algorithm that achieves asymptotic tracking 
is 
$u(t) = y(t).$ Indeed, with this 
choice, the controlled system dynamics are $\dot x(t) = -x(t) +w$, yielding: ${x(t) \to w}$, $y(t) \to -w$, $u(t) \to -w$, and $g(t) \to 0$,
which ensures asymptotic tracking. 
This is a case where $\sigma(w)$ is not injective; in fact, this corresponds to 
a situation where $\sigma:W_\circ \to \emptyset$, since this controller has no 
internal state. Notice that this is not the only controller choice possible, as 
other solutions (for example, using a dynamic state) are possible.~
\QEDB\end{example}

\section{Feedback optimization algorithm design}
\label{sec:alg_design}
Building on the results derived in the previous section, we now turn our 
attention to designing algorithms that solve 
Problem~\ref{prob:feedback_opt_formal}. 
Our controller construction follows the separation principle defined by 
components~\ref{cond:separation1}–\ref{cond:separation2}. Notably, 
component~\ref{cond:separation2} is further structured into two parts: (i) an 
error-feedback mechanism that stabilizes the plant around the center manifold, 
and (ii) a control action that drives the gradient to zero asymptotically 
(see~\eqref{eq:expression_Hc}).

%

Formally, our construction is as follows: given functions $\gamma(w)$ and 
$\pi(w)$ that satisfy~\eqref{eq:dyn_feedback_existence}, we construct a  
controller with state $z=(z_1, z_2),$ $z_1 \in \real^n, z_2 \in \real^p,$ (i.e., the controller state space dimension is $n_c = n + p$)
and dynamics:
\begin{align}\label{eq:controller_dynamics}
    \dot z_1(t) &= f(z_1(t),u(t), z_2(t)) - L_1 e_y(t),\nonumber\\
    \dot z_2(t) &= s(z_2(t)) - L_2 e_y(t),
\end{align}
where
\begin{align*}
u(t) &= \gamma(z_2(t)) + K(z_1(t) - \pi(z_2(t))),\\
e_y(t) &= c(z_1(t), z_2(t))-y(t).
\end{align*}
%
%
Intuitively, the state variable $z_1(t)$ in~\eqref{eq:controller_dynamics} 
operates as a dynamic observer for the plant's state $x(t),$ driven by the 
output-based estimation error $e_y(t)$ (which describes the error 
between the observer's output and true output $y(t)$) with gain $L_1$.
The state $z_2(t)$ acts as a dynamic observer for the exosystem state $w(t),$ 
driven by the output-based estimation error $e_y(t)$ with gain $L_2$.
The control input $u(t)$ is designed to be a static-feedback optimization 
algorithm (cf.~\eqref{eq:expression_Hc}), driven by the estimated states 
$z_1(t)$ and $z_2(t)$ (in place of the true states $x(t)$ and $w(t)$ as 
in~\eqref{eq:expression_Hc}).

In~\eqref{eq:controller_dynamics}, the matrices $K \in \real^{m \times n}, L_1 \in \real^{n \times q},$ and $L_2\in \real^{p \times q}$ are designed such 
that the closed-loop matrices:
\begin{align}\label{eq:controller_dynamics_matrices}
    A+BK && \text{and} &&
    A_L - L C_L,
\end{align}
where $L = [L_1^\tsp, L_2^\tsp]^\tsp$, have eigenvalues\footnote{Notice that 
existence of $K$ and $L$ is guaranteed by 
Assumption~\ref{as:stabilizability_detectability}.} in $\complex_{<}.$
We note that the Hurwitz stability of $A_L - L C_L$ guarantees asymptotic 
stability of the dynamic observer, while the Hurwitz stability of $A + B K$ 
ensures that the state-feedback controller $u(t)$ renders the closed-loop 
system asymptotically stable.

We summarize our algorithm design procedure in 
Algorithm~\ref{alg:dynamic_feedback_design}, and we  prove the effectiveness of 
this method in Theorem~\ref{thm:algorithm_regulation}.

\begin{algorithm}
\caption{Dynamic-feedback optimization algorithm design}
\label{alg:dynamic_feedback_design}
\KwData{Mappings $f(x,u,w),$ $c(x,w),$ $s(w);$
matrices $A,B$ as in~\eqref{eq:jacobians}, $A_L, C_L$ as 
in Proposition~\ref{prop:necessary_conditions_tracking}; mappings $\pi(w), \gamma(w)$ as in Theorem~\ref{thm:existence_dynamic_feedback}}
Let $K$ be any matrix such that $A+BK$ is Hurwitz\;
Let $L$ be any matrix such that $A_L - L C_L$ is Hurwitz\;
Decompose $z=(z_1, z_2), z_1 \in \real^{n}, z_2 \in \real^{p}$ \;
Decompose $L = [L_1^\tsp, L_2^\tsp]^\tsp, L_1 \in \real^{n \times q}, L_2 \in \real^{p \times q}$\;
$G_c(z) \gets \gamma(z_2) + K(z_1 - \pi(z_2))$\;
$F_c(z,y) \gets 
\begin{bmatrix}
    f(z_1,G_c(z), z_2) - L_1 (c(z_1, z_2)-y)\\
    s(z_2) - L_2 ( c(z_1, z_2)-y)
\end{bmatrix}$\;
\KwResult{$F_c(z,y)$, $G_c(z)$, that solve Problem~\ref{prob:feedback_opt_formal}}
\end{algorithm}

\begin{theorem}[\textbf{\textit{Correctness of Algorithm~\ref{alg:dynamic_feedback_design}}}]
\label{thm:algorithm_regulation}
Let Assumptions~\ref{as:lipschitz_convexity}--\ref{as:stabilizability_detectability}
hold, and let the control algorithm~\eqref{eq:controller_equatons} be designed 
following Algorithm~\ref{alg:dynamic_feedback_design}. Then, the controlled 
system~\eqref{eq:interconnected_system}  exactly asymptotically tracks a 
critically trajectory of~\eqref{eq:optimization}.
\QEDB\end{theorem}

\begin{proof}
Because $K$ and $L$ are designed so that $A+BK$ and $A_L - L C_L$ have 
eigenvalues in $\complex_{<}$, the matrices:
\begin{align*}
A+BK && \text{and} &&
\left[\begin{array}{ccc}
A-L_1 C & P-L_1 Q \\
-L_2 C & S-L_2 Q
\end{array}\right]
\end{align*}
have eigenvalues in $\complex_{<}.$ A standard calculation (see the proof of 
Theorem~\ref{thm:existence_dynamic_feedback}) shows that, for any $\Gamma$ and $\Pi,$ the matrix
\begin{align*}
\left[\begin{array}{ccc}
A & B K & B (\Gamma - K \Pi) \\
L_1 C & A+B K-L_1 C & P+B (\Gamma - K \Pi)-L_1 Q \\
L_2 C & -L_2 C & S-L_2 Q
\end{array}\right]
\end{align*}
also has eigenvalues in $\complex_{<}$.
It immediate to see that this matrix is the Jacobian of the dynamics
\begin{align*}
\dot x(t) &= f(x(t),G_c(z(t)), 0), \nonumber \\
\dot z(t) &= F_c(z(t), c(x(t),G_c(z(t))) ), 
\end{align*}
when $F_c(z,y)$ and $G_c(z,y)$ are constructed according to
Algorithm~\ref{alg:dynamic_feedback_design}, when letting
\begin{align}\label{eq:jacobians_pi_gamma}
\Gamma \defeq \left[\frac{\partial \gamma}{\partial w}\right]_{w=0} && \text{and} &&  \Pi \defeq \left[\frac{\partial \pi}{\partial w}\right]_{w=0}.
\end{align}
This proves that~\ref{def:tracking_a} holds. To show that~\ref{def:tracking_b} 
also holds, observe that the construction of $G_c(z)$ and $F_c(z,y)$ 
in Algorithm~\ref{alg:dynamic_feedback_design} implies the the fulfillment 
of~\eqref{eq:dyn_feedback} with 
$$\sigma(w) = (\pi(w), w).$$
Hence, \ref{def:tracking_b} follows by application of 
Theorem~\ref{thm:characterization_dynamic_feedback}.
\end{proof}

We conclude this section by illustrating the design procedure on a quadratic 
optimization problem with a linear plant.

\begin{example}[\textbf{\textit{Illustration of the design procedure for 
dynamic-feedback optimization algorithms}}]
\label{ex:linear_quadratic_pt4}
Consider the problem discussed in 
Examples~\ref{ex:linear_quadratic}--\ref{ex:linear_quadratic_pt3}.
Let $\Pi \in \real^{n \times p}, \Gamma \in \real^{m \times p}$ be matrices 
that satisfy the matrix 
identities~\eqref{eq:linear_static_feedback_solvability}. For this problem, 
the design procedure of Algorithm~\ref{alg:dynamic_feedback_design} reads 
as follows:
\begin{enumerate}
\item Let $K$ be any matrix such that $A+BK$ is Hurwitz;
\item Let $L$ be any matrix such that $A_L - L C_L$ is Hurwitz;
\item Decompose $z=(z_1, z_2), z_1 \in \real^{n}, z_2 \in \real^{p}$;
\item Decompose $L = [L_1^\tsp, L_2^\tsp]^\tsp, L_1 \in \real^{n \times q}, L_2 \in \real^{p \times q}$;
\item Design the map $G_c(z)$ as follows:
\begin{align*}
G_c(z) &= \Gamma z_2  + K(z_1 - \Pi z_2)\\
&= C_c z,
\end{align*}
with $C_c = \begin{bmatrix} K & \Gamma - K \Pi\end{bmatrix};$
\item Design the map $F_c(z,y)$ as follows:
\begin{align*}
F_c(z,y) &= 
\begin{bmatrix}
    A z_1 + B C_c z +Pz_2 -L_1(Cz_1 + Qz_2-y)\\
    S z_2 - L_2 ( C z_1 +Qz_2 -y)
\end{bmatrix}\\
&= A_c z + B_c y,
\end{align*}
with
\begin{align*}
A_c &= \begin{bmatrix}
A+B K-L_1 C & P+B (\Gamma - K \Pi)-L_1 Q \\
-L_2 C & S-L_2 Q
\end{bmatrix},\\
B_c &= 
\begin{bmatrix}
    L_1 \\ L_2
\end{bmatrix}.\QEDBB
\end{align*}
\end{enumerate}
\end{example}




\section{Extensions to constrained problems}\label{sec:extensions}

The optimization problem~\eqref{eq:optimization} contains only an equilibrium constraint, which we assume can be eliminated using the steady-state map $x = h(u,w)$ (cf. Assumption~\ref{as:steadyStateMap}) to obtain the unconstrained problem~\eqref{eq:optimization_unconstrained}. We now discuss extensions to optimization problems with more general constraints. To that end, consider the equality-constrained problem
\begin{alignat*}{2}
    &\text{minimize}\quad && \phi(u,w(t)) \\
    &\text{subject to}\quad && \psi_i(u,w(t)) = 0, \quad i=1,\ldots,r,
\end{alignat*}
where the constraint functions $\psi_i(u,w(t))$ may depend on the control input and disturbance. The associated Lagrangian is
\[
    L(u,\lambda,w(t)) = \phi(u,w(t)) + \sum_{i=1}^r \lambda_i\,\psi_i(u,w(t)),
\]
where $\lambda_i$ is the Lagrange multiplier associated with the $i\textsuperscript{th}$ equality constraint. A pair $(u,\lambda)\in\real^m\times\real^r$ is said to be a saddle-point of the Lagrangian if, for all $(\bar u,\bar\lambda)\in\real^m\times\real^r$,
\[
    L(u,\bar\lambda,w(t)) \leq L(u,\lambda,w(t)) \leq L(\bar u,\lambda,w(t)). 
\]
For any such saddle-point, if strong duality holds, $u$ is primal optimal, $\lambda$ is dual optimal, and the optimal duality gap is zero. Moreover, the gradient of the Lagrangian (assuming it exists) is zero at any saddle-point. It follows from the derivations in the previous sections that the 
gradient-feedback and parameter-feedback algorithms can be directly applied to 
seek a stationary point of the Lagrangian function by replacing the variable $u$ 
with the extended decision variable $\tilde u = (u, \lambda)$ and by considering 
an augmented loss function in~\eqref{eq:optimization_unconstrained} defined as
$\phi(\tilde u, w) = L(u, \lambda, w).$
Notice that, if the critical point computed 
by the controller is also a saddle-point and strong 
duality holds, then it is also a solution to the equality-constrained problem. When strong duality does 
not hold, however, such a saddle-point may not correspond to an optimizer~\cite[Ch.~5]{boyd2004}.

\section{Application robotic balancing control}
\label{sec:simulations}

In this section, we demonstrate the effectiveness of the methods in 
controlling an unstable system. We present two sets of simulations: the 
first considers a quadratic cost function, while the second explores a more 
general cost formulation.

\subsection{Quadratic cost}
We begin by illustrating our approach on the balancing robot presented in 
Section~\ref{sec:segway_robot}. 
Letting $x_1 = \theta$ and $x_2 = \dot \theta,$ a state-space model~\eqref{eq:plant} 
for this robot reads as:
\begin{align}\label{eq:state_space_robot}
\dot x_1(t) &= x_2(t),\nonumber\\
\dot x_2(t) &= \alpha \sin x_1(t) -\beta x_2(t) - \gamma u(t) \cos x_1(t) + \eta w_x(t),\notag\\
y(t) &= x_2(t) + w_y(t),
\end{align}
where $\alpha=\frac{m g \ell}{J_e}, \beta = \frac{k\ell^2}{J_e}, \gamma = \frac{m \ell}{J_e},$  and $\eta = \frac{1}{Je}.$ For our simulations, we used 
sinusoidal signals: $w_x(t) = \rho_x \cos(\bar \omega_x t),$ $\rho_x = 1$ [rad/s${}^2$], $\bar  \omega_x = 1$ [rad/s] and $w_y(t) = \rho_y \cos(\bar  \omega_y t),$ $\rho_y = 0.5$ [rad/s${^2}$], $\bar  \omega_y = 10$ [rad/s]. 
These signals are generated by an exosystem~\eqref{eq:exosystem} with state $w=(w_1, w_2,w_3,w_4) \in \real^4,$ vector field $s(w)=S w$ with
\begin{align}\label{eq:matrixS_im}
    S = \begin{bmatrix}
        0 & 1 & 0 & 0\\
        -\bar \omega_x^2 & 0 & 0 & 0\\
         0 & 0 & 0 & 1 \\
        0 & 0 & -\bar \omega_y^2 & 0\\
    \end{bmatrix},
\end{align}
initial condition $w(0)=(\rho_x, 0, \rho_y,0).$ This exosystem model generates 
the desired signals $w_x(t)$ and $w_y(t)$ when letting\footnote{
Recall that the solutions to the dynamical system $\dot w = S w$ with $w = (w_1, w_2)\in \real^2$ and $S=\left[\begin{smallmatrix}0 & 1\\ -\omega^2 &0\end{smallmatrix}\right]$ are:
$w_1(t) = \cos (\omega t) w_1(0) + \sin (\omega t) \frac{w_2(0)}{\omega}$ and 
$w_2(t) = -\omega \sin (\omega t) w_1(0)+  \cos (\omega t) w_2(0)$.
}
$w_x(t) = w_1(t)$ and $w_y(t) = w_3(t).$
It can be verified that a mapping $h(u,w)$ satisfying 
Assumption~\ref{as:steadyStateMap} is given by:
\begin{align*}
h(u,w) = \begin{bmatrix}
-2\tan^\inv \left(\frac{\alpha-\sqrt{\alpha^2- \eta^2 w_1^2  + \gamma^2 u^2}}{\eta w_1+\gamma u}\right)\\
0
\end{bmatrix},
\end{align*}
when $\eta w_1+\gamma u \neq 0$ and $h(u,w) = [\pi, 0]^\tsp$ otherwise.
At first, we consider the following instance of~\eqref{eq:optimization}:
\begin{align}\label{eq:opt_quadratic}
\underset{u \in \real,\,x \in \real^2}{\text{minimize}} ~~~& 
\tfrac{1}{2} \norm{x}^2,\\
\text{subject to:} ~~ &  0 = x_2, \nonumber\\
& 0= \alpha \sin x_1(t) -\beta x_2(t) \notag\\
& \quad\quad\quad\quad - \gamma u(t) \cos x_1(t) + \eta w_x(t),\notag
\end{align}
which formalizes the objective of balancing the robot at the 
vertical position ($x=0$) while rejecting disturbances. 
To determine mappings $\pi(w)$ and $\gamma(w)$ that 
solve~\eqref{eq:dyn_feedback_existence}, we sought an approximate solution 
using polynomial representations:
\begin{align}\label{eq:polynomial_representations_pi_gamma}
\pi(w) = \sum_{\ell=1}^{d_\pi} \langle \psi^\pi_{\ell}, \Theta_\ell({w}) \rangle,&&
\gamma(w) = \sum_{\ell=1}^{d_\gamma} \langle \psi^\gamma_{\ell}, \Theta_\ell({w}) \rangle,
\end{align}
of order $d_\pi$ and $d_\gamma$, respectively.
Here, $\psi^\pi_{\ell}$ and 
$\psi^\gamma_{\ell},$ are $\binom{\ell p}{\ell}$-dimensional vectors of 
coefficients and
$\Theta_\ell (w)$, $\ell \in \naturalpos,$ is an $\ell$-th order polynomial basis 
of functions; formally:
\begin{align*}
\Theta_\ell (w) &= [w_1^\ell, w_1^{\ell-1}w_2, \dots, w_1^{\ell-1} w_p, \\
    & w_1^{\ell-2} w_2^2, w_1^{\ell-2} w_2 w_3, \cdots, w_1^{\ell-2} w_2 w_p,\cdots, w_p^\ell
    ]^\tsp.
\end{align*}
The parameter vectors $\psi^\pi_{\ell}$ and $\psi^\gamma_{\ell}$ have been 
fitted numerically so that~\eqref{eq:dyn_feedback_existence} holds in a 
neighborhood of the origin of $\real^4.$ In our simulations, 
we used $d_\pi = d_\gamma =4$ and \eqref{eq:dyn_feedback_existence} were 
satisfied up to a relative error of $10^{-7.}$
It is worth stressing that, to solve the internal-model 
conditions~\eqref{eq:dyn_feedback_existence}, knowledge of the magnitude of 
the disturbance signals $\rho_x$ and $\rho_y$ is not needed (since these quantities do not appear in the exosystem model~\eqref{eq:matrixS_im}); it is only required 
to know their frequencies $\omega_x$ and $\omega_y$ (which are the only parameters in $S$).
We applied Algorithm~\ref{alg:dynamic_feedback_design}, selecting the controller 
gain $K$ and the observer gain $L$ so that the eigenvalues of $A_L-LC_L$ are uniformly distributed in the intervals  $[-1, -2]$ and $[-2, -3]$, respectively.
Simulation results for this problem are illustrated in Fig.~\ref{fig:segway_simulations}.
The simulation illustrates that the controller is effective in regulating the 
gradient signal $g(t)$ (cf. \eqref{eq:gradient_signal}) to zero, up to a 
numerical error of order $10^{-6}$ (see Fig.~\ref{fig:segway_simulations}(c)); 
the optimal robot configuration corresponds to a situation where the pendulum is 
precisely in the vertical position described by $(x_1,x_2)=(0,0)$ (see Fig.~\ref{fig:segway_simulations}(c)). Notice 
also that the control input that achieves balancing oscillates periodically 
(see Fig.~\ref{fig:segway_simulations}(b)) to cancel out the effect of the 
oscillatory disturbances (see Fig.~\ref{fig:segway_simulations}(a)).

\begin{figure}[t]
\centering \subfigure[]{\includegraphics[width=\columnwidth]{
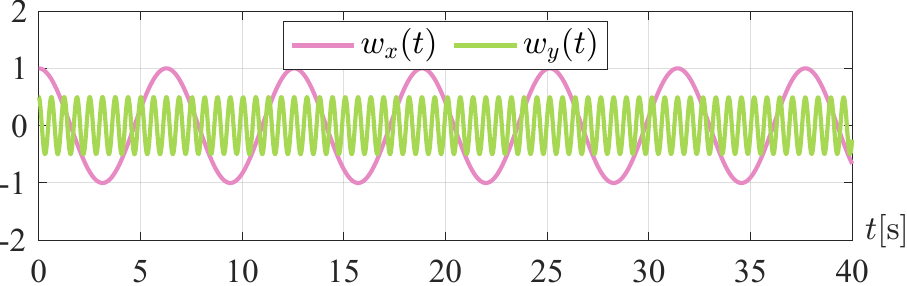}}
\hfill
\centering \subfigure[]{\includegraphics[width=\columnwidth]{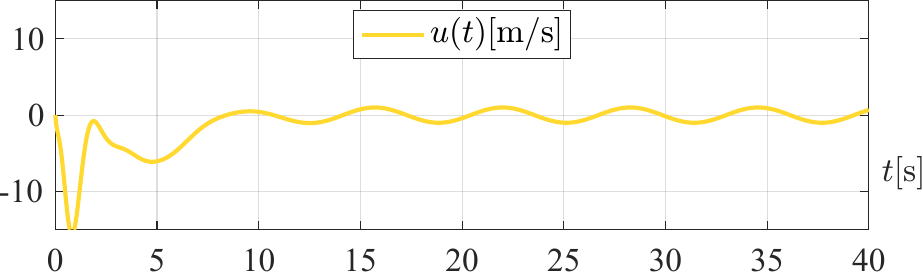}}\\
\centering \subfigure[]{\includegraphics[width=\columnwidth]{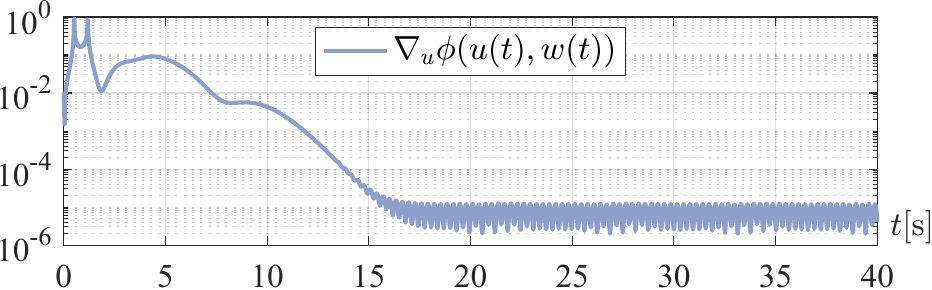}}\\
\centering \subfigure[]{\includegraphics[width=\columnwidth]{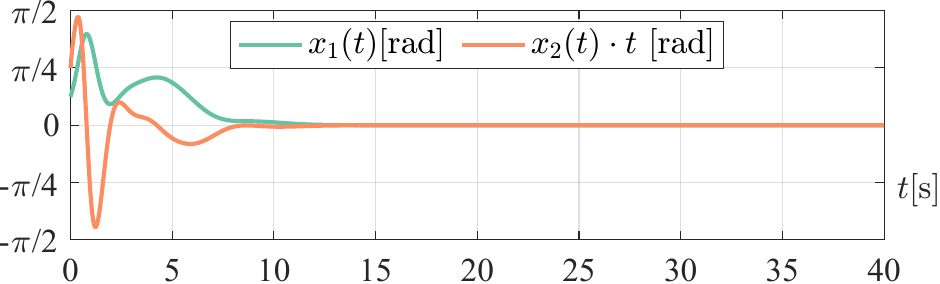}}
\caption{Performance of the proposed control design applied to solve the 
equilibrium-selection problem~\eqref{eq:opt_quadratic} to the balancing robot of 
Fig~\ref{fig:segway}(a) (see model equations~\eqref{eq:state_space_robot}). 
Despite the presence of unmeasurable time-varying disturbances acting on both the 
state and output (Fig.~\ref{fig:segway_simulations}(a)), the controller 
successfully regulates  the gradient error signal to zero -- up to a numerical 
tolerance of order $10^{-6}$ (Fig.~\ref{fig:segway_simulations}(c)). 
This corresponds to 
stabilizing the inverted pendulum in the upright position, characterized by zero 
vertical angle and zero velocity (Fig.~\ref{fig:segway_simulations}(a)). It is 
worth noting that the control input required to achieve this regulation is not 
constant (Fig.~\ref{fig:segway_simulations}(b)). See Fig.~\ref{fig:segway} for the 
parameter values used.
}
\vspace{-.5cm}
\label{fig:segway_simulations}
\end{figure}

\subsection{Logistic regression}
To illustrate the controller performance on a non-quadratic problem, we next replace the loss function in~\eqref{eq:opt_quadratic} by:
 \begin{align}\label{eq:pendulum_logistic}
\phi_0(u,x) =
\tfrac{1}{2} \norm{x}^2 + \tfrac{\kappa}{2}(\log [1+e^{\mu u}]+\log [1+e^{-\mu u}]),
\end{align}
where $\kappa, \mu >0$ (for our experiments, we choose $\mu=0.5$ and 
$\kappa=1$). 
In other words, \eqref{eq:pendulum_logistic} defines a logistic regression 
problem with a time-varying regularization term; intuitively, an optimizer 
of~\eqref{eq:pendulum_logistic} is an equilibrium state for the pendulum model 
such that the robot is as close as possible to being vertically balanced $(x=0),$
while large values (in modulus) of the control input $u$ are penalized by the 
logistic term $\log [1+e^{\mu u}]$. The mappings $\pi(w)$ and $\gamma(w)$ have 
been 
obtained by fitting~\eqref{eq:polynomial_representations_pi_gamma} numerically; 
in this case, in our simulations, \eqref{eq:dyn_feedback_existence} were 
satisfied up to a relative error of $10^{-4.}$
Simulation results for this problem are illustrated in 
Fig.~\ref{fig:segway_logistic}.
By comparing Fig.~\ref{fig:segway}(b) and Fig.~\ref{fig:segway_logistic}(a),
we observe that in the latter case the controller favors control inputs that are 
smaller in magnitude, in line with the above interpretation 
for~\eqref{eq:pendulum_logistic}; as a result of avoiding control inputs of 
large magnitude, the state is not regulated to zero exactly, but the robot 
swings around the vertical position (see Fig.~\ref{fig:segway_logistic}(c));
notice that this configuration does correspond to optimality for this 
optimization problem (see Fig.~\ref{fig:segway_logistic}(b)), in the sense that 
the gradient of \eqref{eq:pendulum_logistic} approaches zero, up to a numerical 
error. 
We reconduct the discrepancy between Fig.~\ref{fig:segway}(c) (where 
$\nabla_u \phi(u(t),w(t)) \approx 0$ up to a numerical error of order $10^{-6}$) 
and Fig.~\ref{fig:segway_logistic}(b) (where 
$\nabla_u \phi(u(t),w(t)) \approx 0$ up to a numerical error of order $10^{-4}$) 
to numerical error in the satisfaction of~\eqref{eq:dyn_feedback_existence}
(relative error of $10^{-7}$ in the former case and $10^{-4}$ in the latter).
We conjecture that the numerical error in achieving 
$\nabla_u \phi(u(t),w(t))\approx 0$ can be further reduced by using 
higher-order polynomials ($d_\pi,d_\gamma>4$) to 
fit~\eqref{eq:dyn_feedback_existence}.

\begin{figure}[t]
\centering \subfigure[]{\includegraphics[width=\columnwidth]{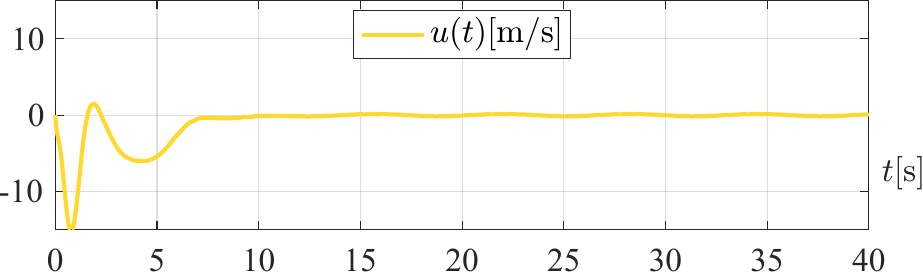}}\\
\centering \subfigure[]{\includegraphics[width=\columnwidth]{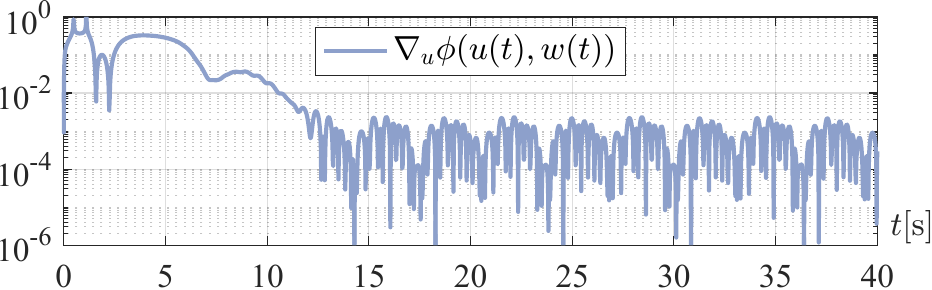}}\\
\centering \subfigure[]{\includegraphics[width=\columnwidth]{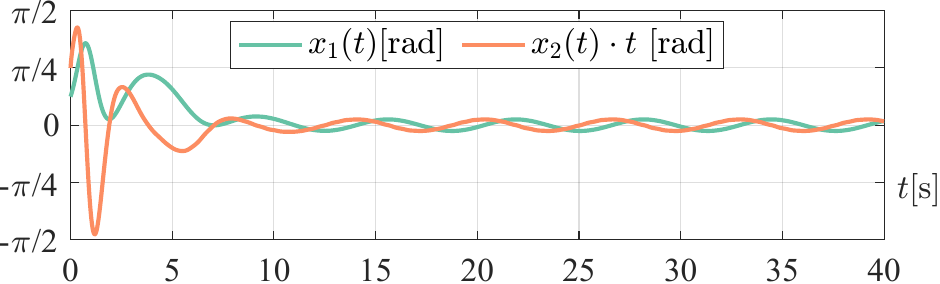}}
\caption{Performance of the proposed control design applied to solve an 
optimization problem with logistic cost~\eqref{eq:pendulum_logistic}. For simplicity of the illustration, we used constant disturbance signals
$\omega_x =1$ [rad/s], $\omega_y=10$ [rad/s]. 
Despite the presence of unknown disturbances acting on both the state and 
output, the controller regulates the gradient error signal to zero -- up to a 
numerical tolerance of order $10^{-4}$ (Fig.~\ref{fig:segway_logistic}(b)), which 
corresponds to striking a balance between stabilizing the inverted pendulum in the 
upright position, characterized by zero vertical angle and zero velocity 
(Fig.~\ref{fig:segway_logistic}(c)) and minimizing the control effort 
(Fig.~\ref{fig:segway_logistic}(a)). See Fig.~\ref{fig:segway} for the parameter 
values.
}
\vspace{-.5cm}
\label{fig:segway_logistic}
\end{figure}

\section{Conclusions}
\label{sec:conclusions}

We have shown that the equilibrium-selection problem studied in feedback 
optimization can be recast as an output regulation problem. This allows us to 
develop methods for exact setpoint tracking even when the optimization problem is 
time-varying. Fundamentally, we show that this requires knowledge of an internal 
model of the temporal variability as well as a particular control design 
architecture --- a fundamental limitation that is proven here for the first time in 
the literature. 
Our algorithm design is novel in the literature, and it combines an output-feedback 
control action that stabilizes the plant with an additional control action that 
drives the system toward the set of critical points of the optimization.
This work opens the opportunity for several directions of future work, including 
an investigation of methods that use inexact internal model knowledge or learn the internal model online, the relaxation of the convexity and smoothness 
assumptions, and an investigation of discrete-time algorithms.

\appendices


%
%

\section{Complementary proofs}
\label{sec:comlemetary_proofs}

\subsection{Proof of Theorem~\ref{thm:existence_parameter_feedback}}
\textit{(Only if)}. Suppose Problem~\ref{prob:parameter_feedback} is solvable, meaning that there exists $H_c(x,w)$ that satisfies \ref{prob:parameter_feedback_a} and \ref{prob:parameter_feedback_b}. 
By Theorem~\ref{thm:characterization_parameter_feedback}, there exists $\pi(w)$ such that the 
conditions in~\eqref{eq:par_feedback} are satisfied. Consequently,~\eqref{eq:par_feedback_existence} is also satisfied by
$\gamma(w)=H_c(\pi(w), w).$

\textit{(If)}. 
We aim to show that, under condition~\eqref{eq:par_feedback_existence}, there 
exists a function $H_c(x, w)$ such that both~\ref{prob:parameter_feedback_a} 
and~\ref{prob:parameter_feedback_b} are satisfied.
First, notice that, by Assumption~\ref{as:stabilizability_detectability}, 
there exists a matrix $K$ such that $(A+BK)$ has eigenvalues in $\complex_{<}$.
Moreover, let $\gamma(w)$ and $\pi(w)$ be two functions such 
that~\eqref{eq:par_feedback_existence} hold, and set:
\begin{align*}
    H_c(x,w) = \gamma(w) + K(x - \pi(w)).
\end{align*}
This choice ensures that condition~\ref{prob:parameter_feedback_a} is 
satisfied: indeed, the Jacobian of $f(x(t),H_c(x(t),0),0)$ is $A+BK$ 
(see~\eqref{eq:closed_loop_static_2} for notation), which has eigenvalues in 
$\complex_{<}.$ Moreover, by construction, $H_c(\pi(w), w) = \gamma(w),$ so 
condition~\eqref{eq:par_feedback_existence_a} reduces 
to~\eqref{eq:par_feedback_a}, and similarly, 
\eqref{eq:par_feedback_existence_b} 
reduces to~\eqref{eq:par_feedback_b}. Thus, by 
Theorem~\ref{thm:characterization_parameter_feedback}, 
condition~\ref{prob:parameter_feedback_b} is also satisfied.

\subsection{Proof of Theorem~\ref{thm:existence_dynamic_feedback}}

\textit{(Only if)}. Suppose Problem~\ref{prob:feedback_opt_formal} is solvable. 
Then, by Theorem~\ref{thm:characterization_dynamic_feedback}, there exists 
functions $\pi(w)$, $\sigma(w)$, and $G_c(z)$ such that the conditions 
in~\eqref{eq:dyn_feedback} are satisfied. Consequently, the conditions 
in~\eqref{eq:dyn_feedback_existence} are also satisfied by letting
$\gamma(w) = G_c(\sigma(w)).$

\textit{(If)}. We aim to show that, under 
condition~\eqref{eq:dyn_feedback_existence}, there exists functions $F_c(z,y)$  
and $G_c(z)$ such that both 
conditions~\ref{def:tracking_a}-\ref{def:tracking_b} are satisfied. 
First of all, notice that, by 
Assumption~\ref{as:stabilizability_detectability}, there exist matrices $K$ 
and $L$ such that (letting
$L = [L_1^\tsp, L_2^\tsp]^\tsp, L_1 \in \real^{n \times q}, L_2 \in \real^{p \times q}$):
\begin{align}
\label{eq:eigs_ABK_observer}
A+BK && \text{and} &&
\left[\begin{array}{ccc}
A-L_1 C & P-L_1 Q \\
-L_2 C & S-L_2 Q
\end{array}\right]
\end{align}
have eigenvalues in $\complex_{<}.$ This implies that, for any pair of 
matrices $\Gamma$ and $\Pi,$ the matrix
\begin{align}\label{eq:jacobian_aux2}
\left[\begin{array}{ccc}
A & B K & B (\Gamma - K \Pi) \\
L_1 C & A+B K-L_1 C & P+B (\Gamma - K \Pi)-L_1 Q \\
L_2 C & -L_2 C & S-L_2 Q
\end{array}\right]
\end{align}
has eigenvalues in $\complex_{<}$. To see this, observe that applying the 
similarity transformation $M \mapsto T^{-1} M T$ with
\begin{align*}
T=\left[\begin{array}{ccc}I_n & 0 & 0 \\ I_n & I_n & 0 \\ 0 & 0 & I_p\end{array}\right]
\end{align*}
to the matrix in~\eqref{eq:jacobian_aux2} gives the block upper-triangular matrix:
\begin{align*}
\left[\begin{array}{ccc}A+B K & B K & B(\Gamma-K \Pi) \\ 0 & A-L_1 C & P-L_1 Q \\ 0 & -L_2 C & S-L_2 Q\end{array}\right],
\end{align*}
whose eigenvalues are determined by~\eqref{eq:eigs_ABK_observer}.

Next, let $\pi(w)$ and $\gamma(w)$ be functions that satisfy the
conditions~\eqref{eq:dyn_feedback_existence}.
With the controller defined by $F_c(z,y)$ and $G_c(z)$ from Algorithm~\ref{alg:dynamic_feedback_design},
it is immediate to verify that 
\eqref{eq:jacobian_aux2}, with $\Gamma$ and $\Pi$ the Jacobian matrices defined in~\eqref{eq:jacobians_pi_gamma}, is the Jacobian matrix of the closed-loop dynamics with zero disturbance~\eqref{eq:autonomous_system_w=0}.
This establishes~\ref{def:tracking_a}. To 
establish~\ref{def:tracking_b}, observe that the construction of 
$G_c(z)$ and $F_c(z,y)$ in Algorithm~\ref{alg:dynamic_feedback_design} 
ensures the fulfillment of~\eqref{eq:dyn_feedback} with 
$\sigma(w) = (\pi(w), w).$
Therefore, condition~\ref{def:tracking_b} follows directly from 
Theorem~\ref{thm:characterization_dynamic_feedback}.

\section{Basic concepts from center manifold theory}
\label{sec:center_manifold}

We now summarize relevant facts on center manifold theory 
from~\cite{JC:81}; see also~\cite{SW:23}. Consider the nonlinear system:
\begin{align}\label{eq:nonlinear_sys_preliminaries}
\dot{x}=f(x)
\end{align}
where $f$ is a $C^k$ vector field defined on an open subset $U$ of $\real^n$. Consider an equilibrium point for $f$, which we take without loss of generality to be zero, i.e., $f(0)=0$. Let
$F=\left[\frac{\partial f}{\partial x}\right]_{x=0},$
denote the Jacobian matrix of $f$ at $x=0$.
Suppose the matrix $F$ has $n^{\circ}$ eigenvalues with zero real part, 
$n^{-}$ eigenvalues with negative real part, and $n^{+}$ eigenvalues with 
positive real part. 
Let $E^{-}, E^{\circ},$ and $E^{+}$ be the (generalized) real eigenspaces of 
$F$ associated with eigenvalues of~$F$ lying on the open left half plane, the 
imaginary axis, and the open right half plane, respectively.
Note that $E^{\circ}, E^{-}, E^{+}$ have dimension  
$n^{\circ}, n^{-}, n^{+}$, respectively and that each of these spaces is 
invariant under the flow of $\dot x = Fx.$
If the linear mapping $F$ is viewed as a representation of the differential 
(at $x=0$) of the nonlinear mapping $f$, 
its domain is the tangent space $T_0 U$ to $U$ at $x=0$, and the three 
subspaces in question can be viewed as subspaces of $T_0 U$ satisfying
$T_0 U=E^{\circ} \oplus E^{-} \oplus E^{+}.$
We refer to~\cite[Sec.~A.II]{MI:01} for a precise definition of $C^k$ 
manifolds; loosely speaking, a set $S \subset U$ is a $C^k$ manifold it can be 
locally represented as the graph of a $C^k$ function.


\begin{definition}[\textbf{\textit{Locally invariant manifold}}]
A $C^k$ manifold $S$ of $U$ is 
locally invariant for~\eqref{eq:nonlinear_sys_preliminaries} if, for 
each $x_{\circ} \in S$, there exists $t_1<0<t_2$ such that the integral curve 
$x(t)$ of~\eqref{eq:nonlinear_sys_preliminaries} satisfying $x(0)=x_{\circ}$ satisfies $x(t) \in S$ for all $t \in (t_1, t_2)$.
\QEDB\end{definition}

Intuitively, by letting $x=(y,\theta)$ and expressing~\eqref{eq:nonlinear_sys_preliminaries} 
as:
\begin{align}\label{eq:nonlinear_sys_preliminaries_b}
\dot{y}=f_y(\theta,y), && \dot{\theta}=f_\theta(\theta,y),
\end{align}
a curve $y = \pi(\theta)$ is an invariant manifold 
for~\eqref{eq:nonlinear_sys_preliminaries_b} if the solution 
of~\eqref{eq:nonlinear_sys_preliminaries_b} with 
$\theta(0)=\theta_\circ$ and $y(0) = \pi(\theta_\circ)$ lies on the curve 
$y=\pi(\theta)$ for $t$ in a neighborhood of $0.$
The notion of invariant manifold is useful as, under certain assumptions, it 
allows us to reduce the analysis of~\eqref{eq:nonlinear_sys_preliminaries} to 
the study of a reduced system in the variable $\theta$ only. The remainder of 
this section is devoted to formalizing this fact.

\begin{definition}[\textbf{\textit{Center manifold}}]
Let $x=0$ be an equilibrium of~\eqref{eq:nonlinear_sys_preliminaries}. A 
manifold $S$, passing through $x=0$, is said to be a center manifold 
for~\eqref{eq:nonlinear_sys_preliminaries} at $x=0$ if it is locally invariant 
and the tangent space to $S$ at 0 is exactly $E^{\circ}$.
\QEDB\end{definition}

Intuitively, the invariant manifold $y=\pi(\theta)$ is a center 
manifold for~\eqref{eq:nonlinear_sys_preliminaries_b} when all orbits of $y$ decay to zero and those of $\theta$ neither 
decay nor grow exponentially. 


In what follows, we will assume that all eigenvalues of $F$ have nonpositive 
real part, i.e., $n^+ = 0$.
When this holds, it is always possible to choose coordinates in 
$U$ such that~\eqref{eq:nonlinear_sys_preliminaries} is
\begin{align}\label{eq:preliminaries_first_order_system}
 \dot{y}=A y+g(y, \theta), && 
 \dot{\theta}=B \theta+h(y, \theta), 
\end{align}
where $A$ is an $n^{-} \times n^{-}$ matrix having all eigenvalues 
with negative real part, $B$ is an $n^{\circ} \times n^{\circ}$ 
matrix having all eigenvalues with zero real part, and the functions $g$ and 
$h$ are $C^k$ functions vanishing at $(y, \theta)=(0,0),$ together with all 
their first-order derivatives.
Because of their equivalence, any conclusion 
drawn for~\eqref{eq:preliminaries_first_order_system} will apply also 
to~\eqref{eq:nonlinear_sys_preliminaries}. 
The following result ensures the existence of a center manifold.

\begin{theorem}[\textbf{\textit{Center manifold existence theorem}}]
\label{thm:existence_center_manifold}
Assume that $n^+ = 0$. There exists a neighborhood 
$V \subset \real^{n^0}$ of $0$ and a class $C^{k-1}$ 
mapping $\pi: V \rightarrow \real^{n^{-}}$ such that the set
$
    S = \{(y, \theta) \in \real^{n^{-}}  \times V: y=\pi(\theta)\},
$
is a center manifold for~\eqref{eq:preliminaries_first_order_system}. 
\QEDB\end{theorem}


By definition, a center manifold 
for~\eqref{eq:preliminaries_first_order_system} passes through $(0,0)$ and 
is tangent to the subset of points with $y = 0$. 
Namely,
\begin{align}\label{eq:mapping_center_manifold}
    \pi(0)=0 \qquad\text{and}\qquad \frac{\partial \pi}{\partial\theta}(0)=0.
\end{align}
Moreover, this manifold is locally invariant 
for~\eqref{eq:preliminaries_first_order_system}: this imposes on the 
mapping $\pi$ the constraint:
%
%
\begin{align}\label{eq:differential_equation_center_manifold}
\frac{\partial \pi}{\partial \theta}(B \theta+h(\pi(\theta), \theta))=A \pi(\theta)+g(\pi(\theta), \theta),    
\end{align}
as deduced by differentiating with respect to time any solution 
$(y(t), \theta(t))$ of~\eqref{eq:preliminaries_first_order_system} on the manifold $y(t)=\pi(\theta(t))$. In other words, any 
center manifold for~\eqref{eq:preliminaries_first_order_system} can 
equivalently be described as the graph of a mapping $y=\pi(\theta)$ satisfying 
the 
partial differential equation~\eqref{eq:differential_equation_center_manifold}, 
with the constraints~\eqref{eq:mapping_center_manifold}.

\begin{remark}
Theorem~\ref{thm:existence_center_manifold} shows existence but 
not the uniqueness of a center manifold. Moreover, (i)
if $g$ and $h$ are $C^k, k \in \naturalpos,$  then
\eqref{eq:preliminaries_first_order_system} admits a $C^{k-1}$
center manifold; (ii) if $g$ and $h$ are 
$C^\infty$ functions, then \eqref{eq:preliminaries_first_order_system} has a $C^k$ 
center manifold for any finite $k$, but not necessarily a $C^\infty$ center 
manifold.
\QEDB \end{remark}

The next result shows that any $y$-trajectory
of~\eqref{eq:preliminaries_first_order_system}, starting sufficiently close to 
the origin converges, as time tends to infinity, to a trajectory that belongs 
to the center manifold.

\begin{theorem}[\textbf{\textit{Center manifold is locally attractive}}]
\label{thm:local_attractivity_manifold}
Assume that $n^+ = 0$ and 
suppose $y=\pi(\theta)$ is a center manifold 
for~\eqref{eq:preliminaries_first_order_system} at $(0,0)$. Let 
$(y(t), \theta(t))$ be a solution
of~\eqref{eq:preliminaries_first_order_system}. There 
exists a neighborhood $U^{\circ}$ of $(0,0)$ and real numbers $M>0$ and $K>0$ 
such that, if $(y(0), \theta(0)) \in U^{\circ}$, then for all $t \geq 0$,
\begin{align*}
\|y(t)-\pi(\theta(t))\| \leq M e^{-K t}\|y(0)-\pi(\theta(0))\|. \tag*{\QEDB}
\end{align*}
\end{theorem}

\renewcommand{\IEEEbibitemsep}{-2pt plus 0.0ex}
\bibliographystyle{myIEEEtran}
\bibliography{BIB/brevalias,BIB/full_GB,BIB/GB,BIB/references}

\end{document}